\pdfoutput=1
\documentclass[11pt,letterpaper]{article}
\usepackage[margin=1in]{geometry}
\usepackage[T1]{fontenc}
\usepackage[utf8]{inputenc}
\usepackage{lmodern}

\usepackage{amsmath,amssymb,amsfonts,amsthm,amsopn}
\usepackage{dsfont,bbm,bm}
\usepackage{thmtools}
\usepackage{thm-restate}

\makeatletter
\@ifundefined{newcounteralias}{}{%
  \renewcommand\thmt@autorefsetup{%
    \@xa\def\csname\thmt@envname autorefname\@xa\endcsname
    \@xa{\thmt@thmname}}%
}
\makeatother

\theoremstyle{plain}
\newtheorem{theorem}{Theorem}
\newtheorem{lemma}[theorem]{Lemma}

\newtheorem{cor}[theorem]{Corollary}

\newtheorem{prop}[theorem]{Proposition}

\theoremstyle{definition}

\newtheorem{defn}[theorem]{Definition}
\newtheorem{example}[theorem]{Example}

\newtheorem{problem}[theorem]{Problem}

\theoremstyle{remark}
\newtheorem{remark}[theorem]{Remark}
\newtheorem{rem}[theorem]{Remark}

\usepackage{graphicx}
\DeclareGraphicsExtensions{.pdf,.png,.jpg}
\usepackage{subfig}
\usepackage{adjustbox,wrapfig}
\usepackage{tikz,qtree}
\usetikzlibrary{calc,3d}
\usepackage{xcolor}

\usepackage{array,multirow,makecell,tabu,booktabs,blkarray}
\usepackage{enumitem}
\setlist[enumerate]{leftmargin=.5in}
\setlist[itemize]{leftmargin=.5in}
\usepackage{algorithm}
\usepackage{algpseudocode}
\usepackage{stackengine}
\newcommand\xrowht[2][0]{\addstackgap[.5\dimexpr#2\relax]{\vphantom{#1}}}
\usepackage{setspace}
\usepackage[normalem]{ulem}
\usepackage{verbatim}

\usepackage[numbers,square]{natbib}
\usepackage{url}
\usepackage[hidelinks]{hyperref}

\makeatletter
\providecommand{\theHALG@line}{}
\renewcommand{\theHALG@line}{\arabic{algorithm}.\arabic{ALG@line}}
\makeatother

\usepackage[capitalize,nameinlink]{cleveref}

\crefname{hypothesis}{Hypothesis}{Hypotheses}
\hypersetup{
  pdftitle={Supermodular Rank: Set Function Decomposition and Optimization},
  pdfauthor={Rishi Sonthalia, Anna Seigal, and Guido Mont\'ufar}
}

\def\checkmark{\tikz\fill[scale=0.4](0,.35) -- (.25,0) -- (1,.7) -- (.25,.15) -- cycle;}

\newcommand{\RBM}{\operatorname{RBM}}
\newcommand{\Vol}{\operatorname{Vol}}

\usepackage{amsmath,amsfonts,bm}
\usepackage{blkarray}

\def\eqref#1{equation~\ref{#1}}

\def\ceil#1{\lceil #1 \rceil}

\def\1{\bm{1}}

\DeclareMathAlphabet{\mathsfit}{\encodingdefault}{\sfdefault}{m}{sl}
\SetMathAlphabet{\mathsfit}{bold}{\encodingdefault}{\sfdefault}{bx}{n}

\DeclareMathOperator*{\argmax}{arg\,max}
\DeclareMathOperator*{\argmin}{arg\,min}

\DeclareMathOperator{\sign}{sign}
\DeclareMathOperator{\Tr}{Tr}

\title{Supermodular Rank: Set Function Decomposition and Optimization}
\author{Rishi Sonthalia\thanks{Math Department, Boston College
  (\href{mailto:rishi.sonthalia@bc.edu}{\texttt{rishi.sonthalia@bc.edu}}).}
\and Anna Seigal\thanks{School of Engineering and Applied Sciences, Harvard University
  (\href{mailto:aseigal@seas.harvard.edu}{\texttt{aseigal@seas.harvard.edu}}).}
\and Guido Mont\'ufar\thanks{Math and Stats \& Data Science Departments, UCLA
  (\href{mailto:montufar@math.ucla.edu}{\texttt{montufar@math.ucla.edu}}).}}
\date{}

\begin{document}

\maketitle

\begin{abstract}
We define the supermodular rank of a set function. This is the smallest number of terms needed to decompose it into a sum of supermodular functions. The supermodular summands are defined with respect to different partial orders. We characterize the maximum possible value of the supermodular rank and describe the functions with fixed supermodular rank. We analogously define the submodular rank. We use submodular decompositions to optimize set functions. Given a bound on the submodular rank of a set function, we formulate an algorithm that splits an optimization problem into submodular subproblems. We show that this method improves the approximation ratio guarantees of several algorithms for monotone set function maximization and ratio of set functions minimization at a computation overhead that depends on the submodular rank. {Finally, we show that the relative volume of the cone of submodular functions on $\{0,1\}^n$ is at most $0.85^{2^n}$, which is much smaller than the naive upper bound of $1/2^{n-1}$. }
\end{abstract}

\section{Introduction}

The optimization of set functions---functions that are defined over families of subsets---is encountered in image segmentation~\cite{1316848}, clustering~\cite{NIPS2005_b0bef4c9}, feature selection~\cite{JMLR:v13:song12a}, and data subset selection~\cite{pmlr-v37-wei15} amongst many more examples. Brute force optimization is often not viable since the individual function evaluations may be expensive and the search space has exponential size.
Therefore, one commonly relies on optimization heuristics that work for functions with a particular structure. A classic function structure is supermodularity, which for a function on a lattice\footnote{I.e., a poset where any two elements $x,y$ have a greatest lower bound $x\wedge y$ and a least upper bound $x\vee y$.} requires that $f(x)+f(y) \leq f(x\wedge y) + f(x\vee y)$. A function is submodular if its negative is supermodular. Submodularity can be interpreted as a ``diminishing returns'' property. Submodularity and supermodularity can be used to obtain guarantees for greedy optimization of set functions in a similar way as convexity and concavity are used to obtain guarantees for gradient descent/ascent. 

Well-known examples of results using supermodularity are the results of \citet{Nemhauser1978AnAO} for greedy maximization of a monotone submodular set function subject to cardinality constraints and those of \citet{Clinescu2011MaximizingAM} for arbitrary matroid constraints. Refinements of these results have been obtained for instance in \cite{CONFORTI1984251,filmus2012tight,sviridenko2017optimal}. As many set functions of interest are not submodular or supermodular, relaxations have been considered, such as the submodularity ratio~\cite{das2011submodular}, generalized curvature \cite{ bian2017guarantees, 10.5555/2634074.2634180, CONFORTI1984251, gatmiry2018non}, weak submodularity~\cite{chen2018weakly,pmlr-v119-halabi20a}, $\epsilon$-submodularity~\cite{JMLR:v9:krause08a}, submodularity over subsets~\cite{10.5555/1347082.1347101}, or bounds by submodular functions~\cite{NIPS2016_81c8727c}. We will discuss some of these in more detail in Section~\ref{sec:set-function-optimization}. 

Supermodular functions are defined by linear inequalities. Hence, they comprise a polyhedral cone called the supermodular cone. The inequalities are called \emph{imset inequalities} in the study of conditional independence structures and graphical models~\cite{studeny2010probabilistic}. They impose non-negative dependence on conditional probabilities. The supermodular cone has been a subject of intensive study~\cite{matus1999,studeny2001}, especially the characterization of its extreme rays~\cite{2011cones,Studeny}, which remains an open problem. Supermodular inequalities appear in semi-algebraic descriptions of probabilistic graphical models with latent variables, such as mixtures of product distributions~\cite{allman2015tensors} 
and restricted Boltzmann machines \cite{Seigal2018MixturesAP}. 

We propose a new approach to define the complexity of a set function; i.e., a real-valued function 
on $2^{[n]}\cong \{0,1\}^{n}$. 
Specifically, we define the \emph{supermodular rank} of a set function. 
Based on this, we present a new algorithm for set function optimization via submodular decompositions. 
This lets us take guarantees for low-complexity functions and apply them to higher-complexity ones. We also obtain the inequalities defining Minkowski sums of supermodular cones, which 
are of interest in the description of latent variable graphical models.

\paragraph{Main contributions}  
\begin{itemize}
\item We introduce the notion of supermodular rank for set functions (Definition~\ref{def:supermodular-rank}). 
The functions of supermodular rank at most $r$ comprise a union of Minkowski sums of at most $r$ supermodular cones. We characterize the facets of these sums (Theorem~\ref{cor:facets-minkowski-sums}) and find the maximum supermodular rank (Theorems~\ref{thm:rank}) and maximum elementary supermodular rank (Theorem~\ref{prop:maximum-elementary-rank}).
\item { 
We show that the relative volume of the set of submodular functions is at most $0.85^{2^n}$ (Section~\ref{app:computing-volume}). 
This illustrates that, though the set of submodular functions on $\{0,1\}^n$ has full dimension, it is miniscule, which helps to explain why there exist algorithms with strong guarantees for such functions.} 
\item We show that the supermodular rank decomposition provides a grading of set functions that is useful for obtaining optimization guarantees. We propose the \textsc{r-split} and \textsc{r-split ratio} algorithms for monotone set function and the ratio of set functions optimization (Algorithms~\ref{alg:modified} and \ref{alg:modified-ratio-1}). These trade-off between computational cost and accuracy, with theoretical guarantees (Theorems~\ref{thm:main-result}, \ref{thm:main-result2a} and \ref{thm:main-result2}) 
that improve on previous guarantees for greedy algorithms based on approximate submodularity (Tables~\ref{tab:max-compare} and~\ref{tab:max-compare-ratio}). We also provide a lower bound for the complexity of optimizing an elementary submodular rank-$(r+1)$ function (Theorem \ref{thm:lower}).
\item We describe a procedure to compute low supermodular rank approximations of functions via convex optimization (Section~\ref{sec:experimental-results}). 
\item Experiments illustrate that our methods are applicable in diverse settings (Section~\ref{sec:experimental-results}).
\end{itemize}
{Our algorithms can be deployed with an ansatz for the rank. However, one limitation is that in order to obtain the theoretical guarantees we need to know the (elementary) submodular rank of the set function that we are given. 
Except for some specific families of functions, currently we do not know how to compute the (elementary) submodular rank of a set function better than with a naive brute-force algorithm. It is likely to be NP-Hard to compute without additional prior information. With structural knowledge about the functions, it can be tractable, as in Examples~\ref{eg:clique} and \ref{eg:weighted-maxcut}.}

\section{Supermodular Cones} 

We introduce our setup and describe basic properties. 

\begin{defn}
    Let $(X,\preceq)$ be a partially ordered set (poset). Given two elements $x,y \in X$, 
    \begin{enumerate}
        \item The \emph{greatest lower bound} $x \wedge y$ is a $z \in X$ such that $z \prec x,y$ and $\forall w \prec x,y$, $w \preceq z$. 
        \item The \emph{least upper bound} $x \vee y$ is a $z \in X$ such that $x,y \prec z$ and $z \preceq w$ for all $x,y \prec w$. 
    \end{enumerate}
    \emph{Lattices} are posets for which any two elements have a least upper and a greatest lower bound.
\end{defn}

\begin{example}[Partially ordered sets] 
        Let $X = X_1 \times \cdots \times X_n$, where $(X_i, \le_i)$ are linearly ordered spaces for $i \in \{1, \ldots, n\}=:[n]$. 
        For $x, y \in X$, let $x \prec y$ if and only if $x_i \le_i y_i$, for $i \in [n]$ and $x \neq y$, where $\le_i$ is the linear ordering on $X_i$. Then $x \wedge y = (\min_{\le_1}(x_1, y_1), \ldots, \min_{\le_n}(x_n,y_n))$ and $x \vee y = (\max_{\le_1}(x_1, y_1), \ldots, \max_{\le_n}(x_n,y_n))$.
\end{example}

We now define supermodular functions. 

\begin{defn} \label{def:supermodular} 
    Let $X$ be a lattice.
    A function $f \colon X \to \mathbb{R}$ is \emph{supermodular} if
        \begin{equation} \label{eq:supermodular}
            \forall\ x,y \in X, \ \ \ \ f(x) + f(y) \le f(x\wedge y) + f(x \vee y).
        \end{equation} 
    The function $f$ is \emph{submodular} (resp.\ \emph{modular}) if $\leq$ in (\ref{eq:supermodular}) is replaced by $\ge$ (resp.\ $=$).
\end{defn}

\begin{example}
    Let $X$ be the poset of all subsets of a set $S$ ordered by inclusion. 
    Then $f(A) = |A|$ is a modular function. If $g$ is concave, then $f(A) = g(|A|)$ is submodular and if $h$ is convex, then $f(A) = h(|A|)$ is supermodular. 
\end{example}
We will be concerned with a specific form of supermodularity known as $\pi$-supermodularity. 

\begin{defn} 
    \label{def:pi_supermodularity}
    Fix $X=\{0,1\}^n$ and consider a tuple $\pi=(\pi_1,\ldots, \pi_n)$ of linear orders $\pi_i$ on $\{0,1\}$. Define the partial order on $X$ as the product of the linear orders. That is, for any $x,y\in X$, we have $x\leq_\pi y$ if and only if $x_i\leq_{\pi_i} y_i$ for all $i\in[n]$. 
    A function that is supermodular with respect to $\pi$ is called \emph{$\pi$-supermodular}. 
\end{defn}

\begin{example}
    Consider 
    $X = \{(0,0), (1,0),(0,1),(1,1)\}$. Suppose 
    $\pi_1$ and $\pi_2$ are both the  
    standard linear ordering on $\{0,1\}$. 
    Then
    $(0,0)$ is the smallest element, $(1,1)$ is the biggest, and $(1,0)$ and $(0,1)$ are intermediate incomparable elements of $X$. 

    Now suppose $\pi_1$ is a transposition, so that $1 \le_{\pi_1} 0$ for the first coordinate. In this case $(1,0)$ is the smallest, $(0,1)$ is the biggest, and $(0,0)$ and $(1,1)$ are intermediate incomparable elements of $X$. With this lattice, $f$ is $\pi$-supermodular if and only if 
    \[
        f((0,0)) + f((1,1)) \le f((1,0)) + f((0,1)). 
    \]
\end{example}

For fixed $X$, the condition of $\pi$-supermodularity is defined by requiring that certain homogeneous linear inequalities hold. Hence the set of $\pi$-supermodular functions on $X$ is a convex polyhedral cone. We denote this cone by $\mathcal{L}_{\pi}\subseteq \mathbb{R}^X$. 
Two product orders generate the same supermodular cone if and only if one is the total reversion of the other. Hence there are {$\frac{1}{2} \prod_i |X_i|! = 2^{n-1}$} distinct cones $\mathcal{L}_\pi$ \cite{allman2015tensors}. The cone of $\pi$-\emph{submodular} functions on $X$ is $-\mathcal{L}_{\pi}$. 

When $X = \{0,1\}^n$, the description of the supermodular cone in \eqref{eq:supermodular} involves $\binom{2^n}{2}$ linear inequalities, one for each pair $x,y\in X$, {$x\neq y$}. 
However, the cone can be described using just $\binom{n}{2}2^{n-2}$ facet-defining inequalities~\cite{Kuipers2010AGO}. These are the elementary imset inequalities~\cite{studeny2010probabilistic}, which compare $f$ on elements of $X$ that take the same value on all but two coordinates. Identifying binary vectors of length $n$ with their support sets in $[n]$, 
these are
\begin{equation} \label{eq:imset}
      f(z \cup \{i\}) + f(z \cup \{j\}) \le f(z) + f(z \cup \{i,j\}),
\end{equation}
where $i,j \in [n]$, $i \neq j$ and $z \subseteq [n] \setminus \{i,j\}$. We partition the $\binom{n}{2}2^{n-2}$ elementary imset inequalities into $\binom{n}{2}$ sets of $2^{n-2}$ inequalities, as follows.

\begin{defn} \label{def:facets} 
    For fixed $i,j\in[n]$, $i\neq j$, we collect all the inequalities in \eqref{eq:imset} 
    into matrix notation as $A^{(ij)}f \geq 0$. 
    We call $A^{(ij)}\in\mathbb{R}^{2^{n-2}\times 2^n}$ the \emph{$(ij)$ elementary imset matrix}. 
\end{defn}

Each row of $A^{(ij)}$ has two entries equal to $1$ and two entries equal to $-1$. 
We give examples to illustrate the elementary imset inequality matrices from Definition~\ref{def:facets}. 

\begin{example}[Elementary imset inequalities] 
    Given $f\colon \{0,1\}^n \to \mathbb{R}$, the elementary imset inequalities are
    \begin{equation} \label{eq:imset_binary}
          f_{\cdots 0 \cdots 1 \cdots} + f_{ \cdots 1 \cdots 0 \cdots} \leq f_{\cdots 0 \cdots 0 \cdots} + f_{\cdots 1 \cdots 1 \cdots} , 
    \end{equation}
    where $f_{\cdots 0 \cdots 1 \cdots}:= f(\cdots 0 \cdots 1 \cdots )$ and 
    the highlighted entries are at 
    positions $i$ and $j$. 
    Fixing $i$ and $j$, one has $2^{n-2}$ inequalities in~\eqref{eq:imset_binary}. For example, if $n=3$, $i = 1$ and $j=2$ then one obtains two inequalities:
    \[ 
        \begin{matrix} f_{010} + f_{100} \leq f_{000} + f_{110} \phantom{.}\\ 
         f_{011} + f_{101} \leq f_{001} + f_{111}.
         \end{matrix} 
    \] 
    Here $A^{(12)}$ is the $2 \times 8$ matrix
    \[ 
        A^{(12)} = 
        \begin{blockarray}{ccccccccc}
            & 000 & 001 & 010 & 011 & 100 & 101 & 110 & 111 \\
            \begin{block}{c(cccccccc)}
                0 & 1 &  & -1 &  & -1 & & 1 & \\
                1 &  & 1 &  & -1 &  & -1 & & 1 \\
            \end{block}
        \end{blockarray}\,.
    \]
\end{example}

The elementary imset inequalities characterization extends to $\pi$-supermodular cones $\mathcal{L}_{\pi}$ with general $\pi$. 

\begin{defn}
    \label{def:the_sign_vector}
    Fix a tuple $\pi = (\pi_1, \ldots, \pi_n)$ of linear orders on $\{0, 1\}$. 
    Its \emph{sign vector} is $\tau = (\tau_1, \ldots, \tau_n) \in \{ \pm 1 \}^n$, where 
    $\tau_i = 1$ if $0 <_{\pi_i} 1$ and $\tau_i=-1$ if $1 <_{\pi_i} 0$. 
\end{defn}

\begin{restatable}[\citet{Kuipers2010AGO}]{lemma}{lemfacets} 
\label{lem:facets} 
Fix a tuple of linear orders $\pi$ with sign vector $\tau$. Then a function $f\colon\{0,1\}^n\to\mathbb{R}$ is $\pi$-supermodular iff $\tau_i\tau_j A^{(ij)} f \ge 0$, for all $i, j \in [n]$ with $i \neq j$. 
\end{restatable}

\begin{example}[Four-bit supermodular functions] 
\label{eg:m42} 
    For $X=\{0,1\}^4$, there are $2^{4-1} = 8$ supermodular cones $\mathcal{L}_\pi$, each given by a sign vector $\tau \in \{\pm 1\}^4$ up to 
    total reversion, namely $(1, 1, 1, 1)$, $(-1, 1, 1, 1)$, $(1, -1, 1, 1)$, $(1, 1, -1, 1)$, $(1, 1, 1, -1)$, $(-1, -1, 1, 1)$, $(-1, 1, -1, 1)$, $(-1, 1, 1, -1)$. Each cone is described by $\binom{4}{2} 
    2^2 = 24$ elementary imset inequalities, which we collect into $\binom{4}{2}$ matrices $A^{(ij)} \in \mathbb{R}^{4 \times 16}$, where the sign of the inequality depends on 
    $\tau_i\tau_j$. We give the signs of the inequalities for three sign vectors. 
    \begin{center}
    \begin{tabular}{c|cccccc}
         $\tau$ & $A^{(12)}f$ & $A^{(13)}f$ & $A^{(14)}f$ & $A^{(23)}f$ & $A^{(24)}f$ & $A^{(34)}f$ \\
         \hline
         $(1, 1, 1,1)$ & $+$ & $+$ & $+$ & $+$ & $+$ & $+$ \\ 
         $(-1, 1, 1,1)$ & $-$ & $-$ & $-$ & $+$ & $+$ & $+$ \\ 
         $(-1, -1, 1,1)$ & $+$ & $-$ & $-$ & $-$ & $-$ & $+$ \\ 
    \end{tabular}
    \end{center}
\end{example}
In Section~\ref{app:computing-volume}, we show that supermodular cones have tiny relative volume, at most $0.85^{2^n}$. 

\section{Supermodular Rank} 
\label{sec:minkowski-sums-of-supermodular-cones}

We describe the facet-defining inequalities of Minkowski sums of $\pi$-supermodular cones. 
Given two cones $\mathcal{P}$ and $\mathcal{Q}$, recall that their Minkowski sum $\mathcal{P} + \mathcal{Q}$ is the set of points $p + q$, where $p \in \mathcal{P}$ and $q \in \mathcal{Q}$. 
For a partial order $\pi$ with sign vector $\tau$, we sometimes write $\mathcal{L}_\tau$ for $\mathcal{L}_\pi$. 

\subsection{Facet Inequalities of Minkowski Sums} 

To begin we prove elementary properties of Minkowski sums of polyhedral cones. 
Given a polyhedral cone $\mathcal{P}$ defined by inequalities $Ax \geq 0$, we denote by $-\mathcal{P}$ the cone defined by the inequalities $Ax \leq 0$. We write $Av > 0$ if every entry of the vector $Av$ is strictly positive.

\begin{prop} 
\label{prop:sum}
    Given matrices $A_i \in \mathbb{R}^{n_i \times N}$, fix the two polyhedral cones
 \begin{gather*}
        \mathcal{P}_1 = \{ v \in \mathbb{R}^N : A_1v \ge 0, A_2v \ge 0, A_3v \ge 0 \} \phantom{.}\\
        \mathcal{P}_2 = \{ v \in \mathbb{R}^N : A_1 v \ge 0, A_2v \le 0, A_4v \ge 0 \}. 
 \end{gather*}
    Assume that the set $S = \{ v \in \mathbb{R}^N: A_1v=0, A_2v > 0, A_3v > 0,  A_4v < 0 \}$ is non-empty. Then $\mathcal{P}_1 + \mathcal{P}_2 = \{v  \in \mathbb{R}^N \colon A_1v \ge 0\}$. 
\end{prop}
\begin{proof}
    Let $z \in \mathcal{P}_1 + \mathcal{P}_2$. Then $z=v+u$ for some $v \in \mathcal{P}_1$ and $u \in \mathcal{P}_2$. Hence $A_1z= A_1(v+u) \ge 0$,
    so $\mathcal{P}_1 + \mathcal{P}_2 \subseteq \{z : A_1z \ge 0\}$. 
    For the reverse containment, take $z$ such that $A_1z \ge 0$. Fix $v^* \in S$. 
    There exists some $\lambda_1 \in \mathbb{R}_{\ge 0}$ such that for all $\lambda > \lambda_1$, we have $A_2(z + \lambda v^*) \ge 0$ and $A_3(z + \lambda v^*) \ge 0$.
    Then $z+\lambda v^* \in \mathcal{P}_1$, since
    \[
        A_1(z+\lambda v^*) = A_1z+\lambda A_1v^* = A_1z \ge 0.
    \]
    Moreover, there exists some $\lambda_2 \in \mathbb{R}_{\geq 0}$ such that for all $\lambda > \lambda_2$, we have that $z - \lambda v^*$ satisfies 
    \[
        A_1(z-\lambda v^*) \ge 0, A_2(z - \lambda v^*) \le 0 \text{ and } A_4(z-\lambda v^*) \ge 0, 
    \]
    so $z-\lambda v \in P_2$. 
    Taking $\lambda > \max ( \lambda_1, \lambda_2)$, we can write $z = \frac{1}{2}(z + \lambda v^*) + \frac{1}{2}(z - \lambda v^*)$ to express $z$ as an element of $\mathcal{P}_1 + \mathcal{P}_2$.
\end{proof}

\begin{example}[Sums of two four-bit supermodular cones] 
\label{eg:rmb42}
    We saw in Example \ref{eg:m42} that there are eight possible $\pi$-supermodular cones $\mathcal{L}_\pi$, each defined by $24$ elementary imset inequalities. 
    Here, we study the facet inequalities of their Minkowski sums. 
    The inequalities defining the Minkowski sum $\mathcal{L}_{(1,1,1,1)} + \mathcal{L}_{(-1,1,1,1)}$ are, 
    as can be computed using polymake~\cite{assarf2014computing}, 
    \[ 
        A^{(23)}f \geq 0, \quad A^{(24)}f \geq 0 , \quad A^{(34)}f \geq 0 .
    \] 
    That is, the facet defining inequalities of the Minkowski sum are precisely the elementary imset inequalities that hold on both individual cones. 
    Similarly we find that 
    $\mathcal{L}_{(1,1,1,1)} + \mathcal{L}_{(-1,-1,1,1)}$ 
    is defined by the inequalities 
    \[ 
        A^{(12)}f \geq 0, \quad A^{(34)}f \geq 0. 
    \] 
    Again, 
    these are just the inequalities present in both cones individually. 
    Notice that $\mathcal{L}_{(1,1,1,1)} + \mathcal{L}_{(-1,1,1,1)}$ is defined by $12$ inequalities while 
    $\mathcal{L}_{(1,1,1,1)} + \mathcal{L}_{(-1,-1,1,1)}$ is defined by eight.
\end{example}
  
We show that the above observation holds in general: sums of $\pi$-supermodular cones are defined by the facet inequalities that are common to all
summands. We do this by showing that the assumptions of Proposition~\ref{prop:sum} hold for sums of supermodular cones. In particular, the Minkowski sum is as large a set as one could expect. {To quantify which facet inequalities are common, we use the following fine-grained notation.}

\begin{defn} 
\label{def:xi-cone}
    Given $\xi \in \{-1,0,1\}^{\binom{n}{2}}$, define  $\mathcal{L}_{\xi} = \{ x \colon \xi_{ij} A^{(ij)} x \ge 0, \text{for all } i \neq j \}$, 
    where $A^{(ij)}$ is the $(i j)$ elementary imset matrix from Definition~\ref{def:facets}. 
    Similarly, define 
    \[
        \mathcal{L}_{\xi}' =  \mathcal{L}_{\xi}  \cap \{ x \colon  A^{(ij)} x  = 0 \, \text{ for all } i \neq j \text{ with } \, \xi_{ij} = 0 \}.  
    \]
\end{defn}
$\mathcal{L}_{\xi}$ and $\mathcal{L}_{\xi}'$ are $\pi$-supermodular cones in the special cases that $\xi_{ij}  = \tau_i\tau_j$ for some $\tau \in \{ \pm 1\}^n$.

\begin{lemma} 
\label{lem:1sign} 
    Fix $\xi \in \{-1,0,1\}^{\binom{n}{2}}$. There exists $z \in \mathcal{L}_{\xi}'$ with $\xi_{ij}A^{(ij)}z > 0$ for all $\xi_{ij} \neq 0$. 
\end{lemma}

\begin{proof}
    First, assume just one entry $\xi_{ij}$ of $\xi$ is non-zero. Let $z^{(ij)} \in \mathbb{R}^{\{0,1\}^{n}}$ have entries $z_\ell^{(ij)} = c_{\ell_i \ell_j}$, for $\ell = (\ell_1,\ldots, \ell_n)\in\{0,1\}^n$. 
    Then all rows of $\xi_{ij} A^{(ij)} z^{(ij)}$ equal $\xi_{ij} ( c_{00} + c_{11} - c_{01} - c_{10})$. We choose the four entries $c_{00}, c_{01}, c_{10}, c_{11}$ so that $\xi_{ij} ( c_{00} + c_{11} - c_{01} - c_{10}) > 0$. Moreover, $A^{(i'j')}z^{(ij)}$ is zero for all other $\{ i',j'\}$, since the value of $z^{(ij)}_\ell$ only depends on $\ell_i,\ell_j$. We conclude by setting $z = \sum z^{(ij)}$, where the sum is over $(i,j)$ with $\xi_{ij} \neq 0$. 
\end{proof}

\begin{prop} 
\label{prop:sum-xi} 
    Fix $\xi^{(1)}, \xi^{(2)} \in \{-1,0,1\}^{\binom{n}{2}}$. The Minkowski sum $\mathcal{L}_{\xi^{(1)}}+ \mathcal{L}_{\xi^{(2)}}$ is cut out by the inequalities common to both summands. That is, $\mathcal{L}_{\xi^{(1)}}+ \mathcal{L}_{\xi^{(2)}} = \mathcal{L}_{\xi}$, where
    \[ 
        \xi_{ij} = \begin{cases} \xi^{(1)}_{ij}, & \xi^{(1)}_{ij} = \xi^{(2)}_{ij} \\ 0, & \text{otherwise}. \end{cases} 
    \]
\end{prop}

\begin{proof}
    Define $\xi'$ by 
    \[
        \xi'_{ij} = \begin{cases} 0,  & \xi^{(1)}_{ij} \cdot \xi^{(2)}_{ij} = 1 \\
        \xi^{(1)}_{ij}, & 
        \xi^{(1)}_{ij} \cdot \xi^{(2)}_{ij} = -1 \\
        \xi^{(1)}_{ij}, & \xi^{(1)}_{ij} \neq 0, \xi^{(2)}_{ij} = 0 \\
        -\xi^{(2)}_{ij}, & \xi^{(2)}_{ij} \neq 0, \xi^{(1)}_{ij} = 0. 
        \end{cases}
    \]
    There exists $v$ in the interior of $\mathcal{L}_{\xi'}'$, by Lemma~\ref{lem:1sign}.
    This $v$ satisfies the assumption from Proposition~\ref{prop:sum} for the polyhedral cones $\mathcal{P}_1 = \mathcal{L}_{\xi^{(1)}}$ and $\mathcal{P}_2 = \mathcal{L}_{\xi^{(2)}}$. 
    The four cases in the definition of $\xi'$ are the four cases  
    $A_1 v = 0, A_2 v > 0, A_3 v > 0, A_4v < 0$ in the set $S$ of Proposition~\ref{prop:sum}. 
\end{proof}

\begin{restatable}[Sums of supermodular cones]{theorem}{facetsminkowskisums}
\label{cor:facets-minkowski-sums}
    Fix a tuple of partial orders $\pi^{(1)},\ldots, \pi^{(m)}$. The Minkowski sum of supermodular cones $\mathcal{L}_{\pi^{(1)}} + \cdots + \mathcal{L}_{\pi^{(m)}}$ is a convex polyhedral cone whose facet inequalities are the facet 
    inequalities common to all $m$ cones $\mathcal{L}_{\pi^{(i)}}$.
\end{restatable}

\begin{proof}
    Let $\tau^{(s)}$ be the sign vector of $\pi^{(s)}$. By Proposition~\ref{prop:sum-xi}, the Minkowski sum in the statement is $\mathcal{L}_\xi$, where 
    \[ 
        \xi_{ij} = 
        \begin{cases} 
            \xi^{(1)}_{ij}, & \tau_i^{(1)} \tau_j^{(1)} = \cdots = \tau_i^{(m)} \tau_j^{(m)} \\ 0, & \text{otherwise}. 
        \end{cases} 
    \]
\end{proof}

\subsection{Supermodular and Submodular Rank} Given our new understanding of the supermodular cone, we define our notion for the complexity of set functions. 
\begin{defn}[Supermodular rank] 
\label{def:supermodular-rank}
    The \emph{supermodular rank} of a function $f\colon \{0,1\}^n \to \mathbb{R}$ is the smallest $r$ such that $f=f_1 + \cdots + f_r$, where each $f_i$ is $\pi$-supermodular for some $\pi$.
\end{defn}

\begin{restatable}[Maximum supermodular rank]{theorem}{theoremrank} 
\label{thm:rank} 
    For $n \ge 3$, the maximum possible supermodular rank of a function $f\colon \{0,1\}^n \to \mathbb{R}$ is $\lceil \log_2 n\rceil + 1$. Moreover, submodular functions in the interior of $-\mathcal{L}_{(1,\ldots,1)}$ have supermodular rank $\lceil \log_2 n\rceil$.
\end{restatable} 

\begin{proof}
    The maximum supermodular rank is the minimal $m$ such that a union of cones of the form $\mathcal{L}_{\pi^{(1)}} + \cdots + \mathcal{L}_{\pi^{(m)}}$ fills the space of functions $f: \{0, 1\}^n \to \mathbb{R}$. 
    Let $\tau^{(k)}$ be the sign vector of partial order $\pi^{(k)}$. 
    We show that the maximum supermodular rank is the smallest $m$ such that there exist partial orders $\pi^{(1)}, \ldots, \pi^{(m)}$ with no pair $i \neq j$ having the same value of the product of signs $\tau_i^{(k)} \cdot \tau_j^{(k)}$ for all $k = 1, \ldots, m$. 
    If there is no such pair $i,j$, then the Minkowski sum fills the space, by Lemma~\ref{lem:facets} and Theorem~\ref{cor:facets-minkowski-sums}.
    Conversely, assume that $\xi_{ij}:= \tau_i^{(k)}\cdot\tau_j^{(k)}$ is the same for all $k$, for some $i \neq j$.
    Let the other entries of $\xi$ be zero. Then the Minkowski sum $\mathcal{L}_{\pi^{(1)}} + \cdots + \mathcal{L}_{\pi^{(m)}}$ is contained in $\mathcal{L}_\xi$, by Theorem~\ref{cor:facets-minkowski-sums}. A union of $\mathcal{L}_\xi$ with $\xi \neq 0$ cannot equal the whole space since $\xi_{ij} A^{(ij)} f \ge 0$ from Lemma~\ref{lem:facets} imposes $2^{n-2}$ inequalities and for $n \geq 3$ there exist functions with different signs for these two or more inequalities, since each inequality involves distinct indices. It, therefore, remains to study the sign vector problem.
    
    \begin{figure}
        \Tree[.$[n]$ [.$[n]_1$ [.$[n]_{1,1}$ $\vdots$ ]
                  [.$[n]_{1,-1}$ $\vdots$ ]]
              [.$[n]_{-1}$ [.$[n]_{-1,1}$ $\vdots$ ]
                    [.$[n]_{-1,-1}$ $\vdots$ ]]]
        \caption{Illustration of the proof of Theorem~\ref{thm:rank}.}
        \label{fig:proof-max-supermodular-rank} 
    \end{figure}
    
    Without loss of generality, let $\tau^{(1)} = (1, \ldots, 1)$. 
    Consider the partition $[n] = [n]_1 \cup [n]_{-1}$, where 
    \[ 
        \tau_i^{(2)} = \begin{cases}
                            1 ,& i \in [n]_1 \\ 
                            -1 ,& i \in [n]_{-1}.
                       \end{cases} 
    \] 
    Let $n_1 := | [n]_1|$. Of the $\binom{n}{2}$ pairs $i \neq j$ there are $n_1(n - n_1)$ with $\tau^{(2)}_i \tau^{(2)}_j = -1$. 
    The quantity $n_1(n - n_1)$ is maximized when $n_1 = \frac{n}{2}$ (for $n$ even) or $n_1 = \frac12 (n \pm 1)$ (for $n$ odd). 
    It remains to consider the pairs $i \neq j$ with $\tau^{(1)}_i \tau^{(1)}_j = \tau^{(2)}_i \tau^{(2)}_j$; that is, $(\tau^{(2)}_i, \tau^{(2)}_j) = (1,1)$ or $(-1, -1)$.
    We have reduced the problem to two smaller problems, each with $\binom{m}{2}$ pairs $i, j$ where $m \le \frac{n+1}{2}$. We choose a partition of $[n]_{1}$ into two pieces, say $[n]_{1,1}$ and $[n]_{1,-1}$, and likewise for $[n]_{-1}$. Define $\tau^{(3)}$ to be $1$ on $[n]_{1,1}, [n]_{-1,1}$ and $-1$ on $[n]_{1,-1},[n]_{-1,-1}$. Then the pairs $i \neq j$ with $\tau^{(1)}_i \tau^{(1)}_j = \tau^{(2)}_i \tau^{(2)}_j = \tau^{(3)}_i \tau^{(3)}_j$ are those with $\{i, j\} \subset [n]_{a,b}$ for some $a,b \in \{-1,1\}$. 
    In a sum of $m$ cones, the set $[n]$ has been divided into $2^
    {m-1}$ pieces. Hence, there is one piece of size at least $\lceil \frac{n}{2^{m-1}} \rceil$, by the Pigeonhole principle. This is at least two for $m \leq \lceil \log_2 n\rceil$. Conversely, choosing a splitting into two pieces of size as close as possible shows that for $m \geq \lceil \log_2 n\rceil + 1$ the set $[n]$ can be divided into pieces of size $1$. {Proposition~\ref{prop:supermodular-rank-submodular} proves the second statement.} 
\end{proof}

\begin{remark}
A function in the interior of the cone $-\mathcal{L}_{(1,\dots,1)}$ satisfies $0 < \kappa < 1$, where $\kappa$ denotes the total curvature. Consequently, submodular functions with non-extreme curvature (i.e., $\kappa<1$) can attain supermodular rank close to the maximum. 
This suggests that total curvature and supermodular rank capture different notions of complexity, much like sparsity and low rank capture different properties of matrices. 
\end{remark}

\begin{prop}[Supermodular rank of submodular functions] 
    \label{prop:supermodular-rank-submodular}
    A strictly submodular function $f \colon 2^{[n]} \to \mathbb{R}$ has supermodular rank  $\lceil \log_2 n \rceil$. 
\end{prop}
\begin{proof} 
    We first show that there exists a submodular function $f$ of supermodular rank at least $\lceil \log_2 n \rceil$.
    Suppose that for all $f \in - \mathcal{L}_{(1,\ldots,1)}$, the supermodular rank of $f$ was at most $\lceil \log_2 n \rceil - 1$. The sum $\mathcal{L}_{(1, \ldots, 1)} + (- \mathcal{L}_{(1, \ldots, 1)})$ is the whole space, by Proposition~\ref{prop:sum}. Then the maximal supermodular rank would be $\lceil \log_2 n \rceil$, contradicting Theorem \ref{thm:rank}. 
    Hence there exists $f \in - \mathcal{L}_{(1,\ldots,1)}$ with supermodular rank at least $\lceil \log_2 n \rceil$.

    A function $g$ in the interior of the submodular cone satisfies $A^{(ij)} g < 0$ for all $\{i,j\}$. Suppose that there exists such a $g$ with supermodular rank less than $\lceil \log_2 n \rceil$. Then there exist $\pi^{(1)}, \ldots, \pi^{(m)}$ for $m < \lceil \log_2 n  \rceil$, such that $g \in \mathcal{L}_{\pi^{(1)}} + \cdots + \mathcal{L}_{\pi^{(m)}}$. There exists $\xi \in \{-1,0,1\}^n$ such that 
    \[
        \mathcal{L}_\xi = \mathcal{L}_{\pi^{(1)}} + \cdots + \mathcal{L}_{\pi^{(m)}},
    \]
    by Proposition \ref{prop:sum}. Hence $g$ satisfies inequalities 
    $
        \xi_{ij} A^{(ij)} g \ge 0
    $,
    by Definition \ref{def:xi-cone}. Therefore $\xi_{ij} \in  \{-1,0\}$ for all $\{i,j\}$. It follows that all submodular functions are in $\mathcal{L}_\xi$, a contradiction since by the first paragraph of the proof there exists a submodular function $f$ of supermodular rank at least $\lceil \log_2 n  \rceil$. 
    
    It remains to show that the supermodular rank of a submodular function is at most $\lceil \log_2 n  \rceil$. 
    That is, we aim to show that $\lceil \log_2 n  \rceil$ cones can be summed to give some $\mathcal{L}_\xi$ with all $\xi_{ij}$ in $\{-1, 0\}$. In the proof of Theorem~\ref{thm:rank}, we counted the number of supermodular cones that needed to give $\mathcal{L}_\xi$ with all $\xi_{ij} = 0$. Here, starting with a partial order with some $\tau_i \tau_j = -1$, instead of all $\tau_i\tau_j = 1$, shows that we require (at least) one fewer cone than in Theorem~\ref{thm:rank}. 
\end{proof}

One interesting question is, can we characterize the functions that have maximum rank? 
By Proposition~\ref{prop:supermodular-rank-submodular}, submodular functions $f$ in the interior of $-\mathcal{L}_{(1,\ldots, 1)}$ do not have full rank. 
They satisfy $A^{(ij)}f < 0$. 
We believe that full rank functions $f$ do not satisfy $A^{(ij)} f > 0$ or $A^{(ij)}f < 0$, for any $i \neq j$. 

{Another question is how to efficiently compute the supermodular rank of a given function. A naive brute force algorithm would iterate over the candidate rank $r=1,\ldots, \lceil \log_2 n \rceil +1$ and evaluate containment in any of the ${2^{n-1} \choose r}$ possible sums of $r$ supermodular cones. 
Below we provide examples where the rank can be bounded.}  

\subsection{Examples of Low Rank Functions} 
We present some examples of functions that have low supermodular rank. 
The first example comes from Restricted Boltzmann Machines (RBMs). 
These are graphical models for probability distributions on $\{0,1\}^n$. 
They have a hyperparameter $m$, the number of hidden nodes. 
The results of \citep{allman2015tensors} show that
when $m=1$ the model can represent a distribution $f(x)$ if and only if $\log f$ is $\pi$-supermodular and satisfies certain polynomial equality constraints. 
When we move to more significant values of $m$, it can be shown that the model can represent a distribution $f$ only if $\log f$ is rank-$m$ supermodular \cite{Seigal2018MixturesAP}. 
Thus, log RBM distributions are functions with bounded supermodular rank. 
We discuss the connections between supermodular rank and probability models further in Appendix~\ref{app:relations-probability-models}.

Building on this, we show that functions represented by shallow neural networks have low rank. Specifically, we have the following.

\begin{prop}
    Let $f$ be a real-valued function on $\{0,1\}^n$  that is the composition of an affine function and a convex function. That is, 
    \[
        f(x) = \phi(wx + c) , 
    \]
    where $w$ is a vector of length $1 \times n$, $c$ is a scalar, and $\phi : \mathbb{R} \to \mathbb{R}$  is convex. Then $f$ is $\sign(w)$-supermodular. 
    A sum of $k$ such functions has supermodular rank at most 
    {$\min(k, \ceil{\log_2 n} +1)$}. 
\end{prop}

\begin{proof} 
    We show that the elementary imset inequalities are satisfied. Such an inequality involves four vectors on a two-dimensional face of the cube $\{0, 1\}^n$. We have two indices $x_i,x_j$ that vary and a fixed value of $x_{[n]\setminus\{i,j\}}$, the vector $x$ restricted to the set $[n]\setminus\{i,j\}$. Let $y$ be the entry of the face where $x_i = x_j = 0$. Letting $c' = c + Ay$, we seek to compare $\phi(c' + w_i) + \phi(c' + w_j)$ with $\phi(c') + \phi(c' + w_i + w_j)$. By the definition of ${\rm sign}(w)$-supermodularity, if both entries $w_i$ and $w_j$ have the same sign, we require $\phi(c' + w_i) + \phi(c' + w_j) \leq \phi(c') + \phi(c' + w_i + w_j)$ while if $w_i w_j \leq 0$, we require $\phi(c' + w_i) + \phi(c' + w_j) \geq \phi(c') + \phi(c' + w_i + w_j).$ The inequalities hold by the fact that $\phi$ is convex. For example, if $w_i,w_j\geq0$, then $c\leq c+w_i, c+w_j\leq c+w_i+w_j$, and we apply the definition of convexity as applied to a comparison of four points.
\end{proof}

Thus, in particular, the functions represented over $\{0,1\}^n$ by a one-hidden layer ReLU neural network with $k$ hidden nodes and positive outer weights are rank-$k$ supermodular. 

\subsection{Elementary Submodular Rank}
\label{sec:elemenetary-supermodular-rank}

We introduce a specialized notion of supermodular rank, which we call the elementary supermodular rank. This restricts to specific $\pi$. Later, we focus on submodular functions, so we phrase the definition in terms of submodular functions. 

\begin{defn}[Elementary submodular functions] 
    If $\pi$ has a unique 
    component with the sign of $\pi_i$ equal to $-1$, then we call $-\mathcal{L}_{\pi}$ an \emph{elementary submodular cone} and we say that a function $f \in -\mathcal{L}_{\pi}$ is $\{i\}$-submodular. 
\end{defn}

\begin{defn}[Elementary submodular rank] 
\label{def:rank}
    The \emph{elementary submodular rank} of a function $f\colon \{0,1\}^n\to\mathbb{R}$ is the smallest $r+1$ such that $f = f_0 + f_{i_1} + \cdots + f_{i_{r}}$, where $f_0 \in -\mathcal{L}_{(1, \ldots, 1)}$ and $f_{i_j}$ is $\{i_j\}$-submodular. 
\end{defn}

Any function can be decomposed as a sum of elementary submodular functions. 

\begin{theorem}\label{thm:supermodular}
Let $f\colon \{0,1\}^n \to \mathbb{R}$. Then there exist $f_0, f_{1}, \ldots, f_{n-1}$ with $f_0\in-\mathcal{L}_{(1,\ldots,1)}$ and $f_{i} \in -\mathcal{L}_{\tau^{(i)}}$, where $\tau_j^{(i)} = -1$ if and only if $j = i$, such that $f = f_0 + \sum_{i=1}^{n-1} f_{i}$. 
\end{theorem} 

\begin{proof}
The sign vector of $-\mathcal{L}_{(1,\ldots,1)}$ is 
$\tau = (1, \ldots, 1)$. 
For all $i,j$ we have $\tau_i \tau_j = 1$.
For the cones $\mathcal{L}_{\tau^{(i)}}$, we have $\tau^{(i)}_i \tau^{(i)}_j = -1$ for all $j \neq i$. Hence there are no facet inequalities common to all $n$ cones. The result then
follows from Proposition~\ref{prop:sum-xi}. 
\end{proof}

\begin{restatable}[Maximum elementary submodular rank]{theorem}{theoremmaximumelementaryrank} 
\label{prop:maximum-elementary-rank}
    For $n \ge 3$, the maximum possible elementary submodular rank of a function $f\colon \{0, 1\}^n \to \mathbb{R}$ is $n$. Moreover, a supermodular function in the interior of $\mathcal{L}_{(1,\ldots,1)}$ has elementary submodular rank $n$. 
\end{restatable}

\begin{proof}
    The maximum elementary submodular rank is at most $n$, by Theorem~\ref{thm:supermodular}. For any $i_1, \ldots, i_r$, the cone $-\mathcal{L}_{(1, \ldots, 1)} + (- \mathcal{L}_{\tau^{(i_1)}}) +  \cdots + (- \mathcal{L}_{\tau^{(i_r)}})$, where $ - \mathcal{L}_{\tau^{(i_j)}}$ is an $(i_j)$-th elementary submodular cone, is defined by the inequalities  
    \[
        A^{(jk)} 
        f 
        \le 0, \quad j,k \not \in \{i_1, \ldots, i_r\} , 
    \]
    see Proposition \ref{prop:sum}. This consists of $2^{n-2}\binom{n-r}{2}$ inequalities. A union of such functions, therefore, cannot equal the full space of functions if $r + 1 \leq n-1$, as in the proof of Theorem~\ref{thm:rank}. 

    If $f$ is in the interior of the supermodular cone, then 
    \[
        A^{(ij)} 
        h 
        > 0, \quad \text{ for all } i \neq j . 
    \]
    For $f$ to be in $-\mathcal{L}_{(1, \ldots, 1)} + (- \mathcal{L}_{\tau^{(i_1)}}) +  \cdots + (- \mathcal{L}_{\tau^{(i_r)}})$, we need there to be no $\{i,j\}$ such that $\{i,j\} \cap \{i_1, \ldots, i_r \} = \emptyset$. Thus, we require $r \geq n-1$. 
\end{proof}

{One naive algorithm to compute the elementary submodular rank is to iterate over $r=1,\ldots, n$ and to test containment in any of the ${n\choose r}$ possible sums of $r$ elementary cones.}

\subsection{Volume Estimates} 
\label{app:computing-volume}

\newcolumntype{L}[1]{>{\raggedright\let\newline\\\arraybackslash\hspace{0pt}}m{#1}}
\newcolumntype{C}[1]{>{\centering\let\newline\\\arraybackslash\hspace{0pt}}m{#1}}
\newcolumntype{R}[1]{>{\raggedleft\let\newline\\\arraybackslash\hspace{0pt}}m{#1}}

\begin{table}[H]
    \caption{Volume estimates.} 
    \label{tab:volume}
    \centering
    \begin{tabular}{c|cccc}
    \toprule
   \multicolumn{5}{c}{Submodular} \\
   \midrule 
     $n$   & rank 1 & rank 2 & rank 3 & rank 4   \\
     \midrule
     3 & 12.5\% & 74.9\% & 100\% & - \\
     4 & 0.0072\% & 5.9\% & 100\% & - \\
        \toprule
   \multicolumn{5}{c}{Elementary Submodular} \\        
      \midrule 
     $n$   & rank 1 & rank 2 & rank 3 & rank 4  \\
     \midrule
     3 & 3.14\% & 53.16\% & 100\% & -  \\
     4 & $6 \cdot 10^{-4}$\% & 0.38\% & 29\% & 100\% \\
    \bottomrule
    \end{tabular}
\end{table}

The set of supermodular rank-$r$ functions on $\{ 0, 1\}^n$ is a union of polyhedral cones in $\mathbb{R}^{2^{[n]}}$. 
We can numerically estimate their relative volume as follows. We list the inequalities for each of the Minkowski sums of supermodular cones, then sample $500{,}000$ random points from a standard Gaussian, and test how many of these points live in the union of the cones. 
Table~\ref{tab:volume} shows these estimates. We see that the volume of the cones decreases with the number of variables $n$ and increases rapidly with the rank $r$. 
For instance, when $n=4$ the volume of the set of submodular rank-$2$ functions, 5.9\%, is nearly $1000$ times larger than the volume of the set of submodular rank-$1$ functions, 0.0072\%. 

\paragraph{Upper bound on the volume of the supermodular cone} 

We use the correspondence in Proposition~\ref{prop:submodular-restriction} to upper bound the volume of the supermodular cone on $n$ variables, as follows. 
There are $2^{n-1}$ distinct $\pi$-supermodular cones, all of the same volume and with disjoint interiors. Hence the volume of $\mathcal{L}_{(1,\ldots, 1)}$ is trivially upper bounded by $2^{-n+1}$.
We show that in fact the volume decreases much faster, at least with a quadratic rate in $n$. 

\begin{prop}
\label{prop:volume-decay}
    The relative volume of $\mathcal{L}_{(1,\ldots, 1)} \subseteq \mathbb{R}^{\{0,1\}^n}$ is bounded above by $0.85^{2^n}$. In particular, it decreases at least with a quadratic rate as $n$ increases. 
\end{prop}

\begin{proof}
    We define $\mathcal{L}_0^{(n)}:=\mathcal{L}_{(1,\ldots, 1)}$ and denote the $i^{th}$ elementary supermodular cone on $n$ variables by $\mathcal{L}_i^{(n)}$. 
    For distinct $i_1,\ldots, i_{r}\in[n]$,  the cones $\mathcal{L}_{0}^{(n)},\mathcal{L}_{i_1}^{(n)}, \ldots, \mathcal{L}_{i_{r}}^{(n)}$ have the same volume and disjoint interiors, and thus 
    \[
        (r+1) \Vol(\mathcal{L}_{0}^{(n)}) \leq \Vol(\mathcal{L}_{0}^{(n)} + \mathcal{L}_{i_1}^{(n)}+\cdots + \mathcal{L}_{i_{r}}^{(n)}) . 
    \]
    A function $f$ in $\mathcal{L}_{0}^{(n)} + \mathcal{L}_{i_1}^{(n)}+\cdots + \mathcal{L}_{i_{r}}^{(n)}$ consists of $2^r$ supermodular pieces on $n-r$ variables. That is, each of the pieces of $f$ lies in $\mathcal{L}_0^{(n-r)}$. Thus $f$ lives in a Cartesian product of $2^r$ cones $\mathcal{L}_0^{(n-r)}$. 
    Conversely, a function $f$ that lives in this product of cones lies in the Minkowski sum $\mathcal{L}_{0}^{(n)} + \mathcal{L}_{i_1}^{(n)}+\cdots + \mathcal{L}_{i_{r}}^{(n)}$. Hence
    \[
        \Vol(\mathcal{L}_{0}^{(n)} + \mathcal{L}_{i_1}^{(n)}+\cdots + \mathcal{L}_{i_{r}}^{(n)}) = \Vol(\mathcal{L}_{0}^{(n-r)})^{2^r}.
    \]
    Taking $r=n-2$, and noting that $\mathcal{L}_0^{(2)}=\frac12$, we obtain
    \begin{equation}
        \Vol(\mathcal{L}_{0}^{(n)}) \leq \frac{2^{-2^{n-2}}}{n-1} \leq 0.85^{2^n}, \quad n\geq 2. 
        \label{eq:vol-decrease}
    \end{equation}
    In particular, the relative volume of $\mathcal{L}_{(1,\ldots, 1)} \subseteq \mathbb{R}^{2^{[n]}}$ decreases at least with quadratic rate. 
\end{proof} 

Although \eqref{eq:vol-decrease} significantly improves the trivial upper bound, it is not clear how tight it is. 
For $n=2$, the expression $\frac{2^{-2^{n-2}}}{n-1}$ in \eqref{eq:vol-decrease} equals 50\%, which agrees with the true volume of $\mathcal{L}_0^{(2)}$. However, for $n=3$ and $n=4$ the values 12.5\% and 2.1\% it provides are larger than the experimentally obtained 3\% and 0.0006\% reported in Table \ref{tab:volume}. 
Nonetheless, the result shows that the relative volume of supermodular functions is tiny in high dimensions. This provides additional support and motivation to study relaxations of supermodularity such as our supermodular rank. 

\section{Set Function Optimization} 
\label{sec:set-function-optimization}

The elementary submodular rank gives a gradation of functions that we now apply in the context of set function optimization. 
We show that the theoretical guarantees that exist for submodular function optimization can be ported over to a broader family of functions, namely functions with low elementary submodular rank. 
The extended applicability of the theoretical guarantees comes with an unavoidable increase in the running time, exponential in $r$. 
The first application is to constrained set function maximization. The second application is ratio of set functions minimization. 

\subsection{Maximization of Monotone Functions} 

\begin{defn}
    A set function $f\colon 2^V\to\mathbb{R}$ is \emph{monotone (increasing)} if $f(A) \! \le \! f(B)$ for all $A \! \subseteq \! B$, \emph{normalized} if $f(\emptyset) = 0$, and \emph{positive} if $f(A) > 0$ for all $A \in 2^V\setminus \{\emptyset\}$. 
\end{defn}

Examples of monotone submodular functions include the {discrete} entropy $S\mapsto H(X_S)$ and the mutual information $S\mapsto I(Y;X_S)$, {when $Y$ renders the components of $X$ independent}. 
Unconstrained maximization of monotone set functions gives an optimum at $S=V$. However, in the constrained setting the problem can be NP-hard.  

\begin{defn} 
    A system of sets $\mathcal{M} \subseteq 2^V$ is a \emph{matroid}, if (i) $S \in \mathcal{M}$ and $T \subset S$ implies $T \in \mathcal{M}$ and (ii) $S,T \in \mathcal{M}$ and $|T|+1 = |S|$ implies $\exists e \in S \setminus T$ such that $T \cup \{e\} \in \mathcal{M}$. The \emph{matroid rank} is the cardinality of the largest set in $\mathcal{M}$. 
\end{defn}

\begin{problem} 
    Let $f$ be a normalized, monotone set function and let $\mathcal{M} \subseteq 2^V$ be a matroid. The $\mathcal{M}$-matroid constrained maximization problem is $\max_{S \in \mathcal{M}} f(S)$. The cardinality constrained problem is the special case with $\mathcal{M}=\{S\subseteq V\colon |S|\leq m\}$, for a given $m \leq |V|$. 
\end{problem} 

\begin{algorithm}[H] 
\caption{\textsc{Greedy}}
\label{alg:greedy}
\small 
\begin{algorithmic}[1]
\Function{Greedy}{$f$, $\mathcal{M}$}
    \State $S_0 = \emptyset$, $F = \{ e \in V : S_0 \cup \{e\} \in \mathcal{M}\} $
    \While{$F \neq \emptyset$}
        \State $e  = \argmax_{e \in F}\Delta(e|S_k)$ 
        \State $S_{k+1} = S_k \cup \{e\}$
        \State $F = \{ e \in V : S_{k+1} \cup \{e\} \in \mathcal{M}\} $
    \EndWhile
    \State \Return $S_k$
\EndFunction
\end{algorithmic}
\end{algorithm}

A natural approach to finding an approximate solution to cardinality constrained monotone set function optimization is by an iteration that mimics gradient ascent. Given $S\subseteq V$ and $e\in V$,  the discrete derivative of $f$ at $S$ in direction $e$ is
\[
    \Delta(e | S) = f(S \cup \{e\}) - f(S). 
\]
\looseness=-1The \textsc{Greedy} algorithm~\cite{Nemhauser1978AnAO} produces a sequence of sets starting with $S_0 = \emptyset$, adding at each iteration an element $e$ that maximizes the discrete derivative subject to $S_i\cup\{e\}\in\mathcal{M}$ until no further additions are possible, see Algorithm~\ref{alg:greedy}. A well-known result by \citet{Nemhauser1978AnAO} shows that \textsc{Greedy} has an approximation ratio (i.e., returned value divided by optimum value) of at least $(1 - e^{-1})$ for monotone submodular maximization with cardinality constraints. 

\begin{theorem}[\citet{Nemhauser1978AnAO}] 
\label{thm:nemhauser78}
Fix a monotone, normalized, submodular function $f \colon 2^V \to \mathbb{R}$ and let $\{S_i\}_{i \ge 0}$ be the greedily selected sets for the cardinality constrained problem. Then for all positive integers 
$\ell\leq m\leq |V|$, 
\[
    f(S_\ell) \ge \left(1 - e^{-\ell/m}\right) \max_{\Omega : |\Omega| \le m} f(\Omega). 
\]
In particular, for $\ell = m$, $f(S_m)$ is a $(1-e^{-1})$ approximation for the optimal solution. 
\end{theorem} 

Building on this, \citet{Clinescu2011MaximizingAM, filmus2012tight} show a polynomial time algorithm exists for the matroid constraint case with the same approximation ratio of $(1-e^{-1})$. 

\begin{theorem}[\citet{Clinescu2011MaximizingAM}] 
\label{thm:matorid-sub} 
There is a randomized algorithm that gives a $(1 - e^{-1})$-approximation (in expectation over the randomization in the algorithm) to the problem of maximizing a monotone, non-negative, submodular function $f\colon 2^{[n]} \to \mathbb{R}$ subject to matroid constraint $\mathcal{M}$ given by a membership oracle. The algorithm runs in $\tilde{O}(n^8)$ time.\footnote{$\tilde{O}$ means we ignore logarithmic factors.} 
\end{theorem}
These guarantees can be refined by measuring how far a function is from being modular.

\begin{defn}[\citet{CONFORTI1984251}]  
\label{def:total-curvature}
    The \emph{total curvature} of a normalized, monotone  submodular set function $f$ is 
    \[
        \hat{\alpha} := \max_{e \in \Omega}\frac{\Delta(e|\emptyset) - \Delta(e|V\setminus \{e\})}{\Delta(e|\emptyset)}, \quad \text{where $\Omega = \{ e \in V \colon \Delta(e|\emptyset) > 0\}$.}
    \]
\end{defn} 
It can be shown that $\hat{\alpha} = 0$ if and only if $f$ is modular and that $\hat{\alpha} \leq 1$ for any submodular~$f$. 
\citet{sviridenko2017optimal} present a \textsc{Non-Oblivious Local Search Greedy} algorithm for which they have the following theoretical results. 

\begin{algorithm}[H]
\caption{\textsc{Non-Oblivious Local Search Greedy}}
\label{alg:non-oblivious}
\small 
\begin{algorithmic}[1]
\Function{Non-Oblivious Local Search Greedy}{$f,\epsilon$}
    \State Define $\ell(A) = \sum_{j \in A} f(V) - f(V \setminus \{j\})$
    \State Define $g(A) = f(A) - \ell(A)$
    \State Define $h(A) = \sum_{B \subseteq A} g(B) \int_0^1 \frac{e^p}{e-1} p^{|B|-1} (1-p)^{|A|-|B|} dp$
    \State Define $\hat{v}_g = \max_{j\in V} g(\{j\}), \hat{v}_\ell = \max_{j \in V} |l(\{j\})|$, and $\hat{v} = \max(\hat{v}_g,\hat{v}_\ell)$.
    \State Define $\psi(A) = (1-e^{-1}) h(A) + \ell(A)$
    \State Let $\delta = \frac{\epsilon}{n} \hat{v}$
    \State Let $S$ be an arbitrary base of the Matroid $\mathcal{M}$
    \While{True}
        \State improved $\gets$ False
        \For{$a \in S, b \in V \setminus S$}
            \State $S' \gets (S \cup \{b\}) \setminus \{a\}$
            \If{$S'$ is a base of $\mathcal{M}$ \textbf{and} $\psi(S') \geq \psi(S) + \delta$}
                \State $S \gets S'$
                \State improved $\gets$ True
                \State \textbf{break for loop}
            \EndIf
        \EndFor
        \If{\textbf{not} improved}
            \State \Return $S$
        \EndIf
    \EndWhile
\EndFunction
\end{algorithmic}
\end{algorithm}

\begin{theorem}[{\citet[Theorem 6.1]{sviridenko2017optimal}}] \label{thm:svidernko} For every $\epsilon > 0$, matroid $\mathcal{M}$, and monotone, non-negative, submodular function $f$ with total curvature $\hat{\alpha}$, \textsc{Non-Oblivious Local Search Greedy} produces a set $S$ with high probability, in $O(\epsilon^{-1} poly(n))$ time 
with 
\[
    f(S) \ge (1 - \hat{\alpha}e^{-1} + O(\epsilon)) \max_{\Omega \in \mathcal{M}} f(\Omega). 
\]
\end{theorem}

\paragraph{Approximately submodular functions} To optimize functions that are not submodular, many prior works have looked at approximately submodular functions. 
We briefly discuss these here. 

\begin{defn}
\label{def:generalizedcurvaturesubmratio}
Fix a non-negative monotone set function $f : 2^V \to \mathbb{R}$.  
    \begin{itemize}[nosep]
        \item The \emph{submodularity ratio}~\cite{das2011submodular,bian2017guarantees,Wang2019MinimizingRO} with respect to a set $X$ and integer $m$ is
        \[
            \gamma_{X,m} := \min_{T \subset X, S \subset V, |S| \le m, S \cap T = \emptyset} \frac{\sum_{e \in S}\Delta(e|T)}{\Delta(S|T)}.
        \]
        We drop the subscripts when $X = V$ and $m = |V|$.
        
        \item The \emph{generalized curvature}~\cite{bian2017guarantees} is the smallest $\alpha$ s.t.\ for all $T, S \in 2^V$ and $e \in S \setminus T$,  
        \[
            \Delta(e | (S \setminus \{e\}) \cup T) \ge (1-\alpha)\Delta(e | S \setminus \{e\}) . 
        \]
    \end{itemize}
\end{defn}

For monotone $f$, we have $\gamma \in [0,1]$, with $\gamma = 1$ if and only if $f$ is submodular; we have $\alpha \in [0,1]$, with $\alpha = 0$ if and only if $f$ is supermodular. 
Thus, if $\alpha = 0$ and $\gamma = 1$, then $f$ is modular. \citet{bian2017guarantees} obtained an $\frac{1}{\alpha}(1-e^{-\alpha \gamma})$ approximation ratio guarantee for the \textsc{Greedy} algorithm on the cardinality constraint problem. 

\begin{theorem}[\citet{bian2017guarantees}] \label{thm:bian} 
Fix a non-negative monotone function $f \colon 2^V \to \mathbb{R}$ with submodularity ratio $\gamma$ and curvature $\alpha$. 
Let $\{S_i\}_{i \ge 0}$ be the sequence produced by the \textsc{Greedy} algorithm. Then for all positive integers $m$, 
\begin{align*}
    f(S_m) &\ge \frac{1}{\alpha}\left[1 - \left(\frac{m - \alpha \gamma}{m}\right)^m\right] \max_{\Omega : |\Omega| \le m} f(\Omega) \\
           &\ge  \frac{1}{\alpha}(1-e^{-\alpha \gamma})\max_{\Omega : |\Omega| \le m} f(\Omega). 
\end{align*}
Further, the above bound is tight for the \textsc{Greedy} algorithm. 
\end{theorem}

For optimization of non-negative monotone $f$ with general matroid constraints, \cite{10.5555/2634074.2634180, chen2018weakly, gatmiry2018non} provide approximation guarantees that depend on $\gamma$ and $\alpha$. 

\begin{theorem}[\citet{gatmiry2018non}] \label{thm:gat1} 
Given a matroid $\mathcal{M}$ with rank $\rho \ge 3$ and a monotone set function $f$ with submodularity ratio $\gamma$, \textsc{Greedy}  
returns a set $S$ such that 
\[
    f(S) \ge \frac{0.4\gamma^2}{\sqrt{\rho\gamma}+1} \max_{\Omega \in \mathcal{M}}f(\Omega) . 
\]
\end{theorem}

\begin{theorem}[\citet{gatmiry2018non}] \label{thm:gat2}
Given a matroid $\mathcal{M}$ and a monotone set function $f$ with curvature $\alpha$, 
\textsc{Greedy} returns a set $S$ such that 
\[
    f(S) \ge \left(1+\frac{1}{1-\alpha}\right)^{-1} \max_{\Omega \in \mathcal{M}}f(\Omega).
\]
\end{theorem} 

\subsection{Bounded Elementary Submodular Rank} 

We formulate an algorithm with guarantees for maximizing set functions with bounded elementary submodular rank. We first give definitions and properties of low elementary submodular rank functions.  

\begin{restatable}{defn}{defpiab}
    \label{def:PiAB} 
    Given $A \subseteq B \subseteq V$, define $\Pi(A,B) := \{ C \subseteq V : C \cap B = A\}$. Given a set function $f : 2^V \to \mathbb{R}$, we let $f_{A,B}$ denote its restriction to $\Pi(A,B)$. {That is $f_{A,B} : \Pi(A,B) \to \mathbb{R}$ is defined by $f_{A,B}(C)=f(C)$ for $C\in \Pi(A,B)$}. 
\end{restatable} 
Note that $\Pi(A,B)\cong 2^{V\setminus B}$. If $B = \{i\}$, then $f_{\{i\},B}$ and $f_{\emptyset,B}$ are the pieces of $f$ on the sets that contain (resp.\ do not contain) $i$. If $|B| = m$, then we have $2^m$ pieces, the choices of $A \subseteq B$. We now state a few useful propositions. 

\begin{prop} 
\label{prop:rank1} 
    Let 
    $f$
    be 
    $\{i\}$-submodular.  
    Then $f_{\{i\},\{i\}}$ and $f_{\emptyset, \{i\}}$ are submodular. 
\end{prop}
\begin{proof} 
    For all $S_1, S_2 \in \Pi(\{i\},\{i\})$, we know that $i \in S_1, S_2$. Hence the linear ordering on the $i^{th}$ coordinate does affect the computation of the least upper bound and greatest lower bound. Thus, $S_1 \wedge S_2 = S_1 \cap S_2$ and $S_1 \vee S_2 = S_1 \cup S_2$. Hence the submodularity inequalities from Definition \ref{def:supermodular} hold. Similarly, for $S_1, S_2 \in \Pi(\emptyset,\{i\})$, we have $i \not\in S_1, S_2$. Hence $S_1 \wedge S_2 = S_1 \cap S_2$ and $S_1 \vee S_2 = S_1 \cup S_2$. 
\end{proof}

\begin{restatable}{prop}{propsubmodularrestriction} 
\label{prop:submodular-restriction}
    A set function $f$ has elementary submodular rank $r+1$, with decomposition $f = f_0 + f_{i_1} + \cdots + f_{i_{r}}$, if and only if $f_{A,B}$ is submodular for 
    each $A \subseteq B = \{i_1, \ldots, i_{r}\}$. 
\end{restatable}
\begin{proof}
    If $f$ has such a decomposition, then the fact that the pieces $f_{A,B}$ are submodular follows from Proposition \ref{prop:rank1}, using that a sum of $(1, \ldots, 1)$-submodular functions is $(1, \ldots, 1)$-submodular. For the converse, assume that $f_{A,B}$ is submodular for (without loss of generality) $B = \{1, \ldots, r\}$ and any $A \subseteq B$. Then $f \in \mathcal{L}_{\xi}$, where $\xi_{ij} = -1$ for all $i \neq j$ with $i, j \geq r+1$. The cone $\mathcal{L}_\xi$ is a sum of $-\mathcal{L}_{(1, \ldots, 1)}$ and the $\{ i\}$-submodular cones for all $i \leq r$.
\end{proof}
The pieces $f_{A,B}$ behave well in terms of their submodularity ratio and curvature. 

\begin{restatable}{prop}{proprestrictedalphagamma} 
\label{prop:restricted-alpha-gamma}
    If $f$ is a set function with submodularity ratio $\gamma$ and generalized curvature $\alpha$, then $f_{A,B}$ has submodularity ratio $\gamma_{A,B} \ge \gamma$ and generalized curvature $\alpha_{A,B} \le \alpha$. 
\end{restatable}
\begin{proof}
    This follows from Definitions~\ref{def:generalizedcurvaturesubmratio} and \ref{def:alpha-gamma-min-max}, as in restriction $f_{A,B}$ we have fewer sets $S, T$, in the definition, for which 
    the inequality in the definitions needs to hold. 
\end{proof}

\begin{algorithm}[H]
\caption{\textsc{r-Split}}
\label{alg:modified}
\begin{algorithmic}[1]
\Function{r-Split}{$f$, $r$, $\mathcal{A}$ - subroutine}
    \For{$A \subseteq B \subseteq V$ with $|B| = r$}
        \State run $\mathcal{A}$ on $f_{A,B}$ 
    \EndFor
    \State run $\mathcal{A}$ on $f$\;
    \State \Return Best seen set 
\EndFunction
\end{algorithmic}
\end{algorithm}

We propose the algorithm \textsc{r-Split}, see Algorithm~\ref{alg:modified}. The idea is to run \textsc{Greedy}, or another optimization method, on the pieces $f_{A,B}$ separately and then choose the best solution.  

\begin{defn} \label{def:alpha-gamma-min-max} 
    We define $\displaystyle \alpha_r := \min_{B \subseteq V, |B| \le r} \left( \max_{A \subseteq B} \alpha_{A,B} \right) \text{ and }
        \gamma_r := \max_{B \subseteq V, |B| \le r} \left( \min_{A \subseteq B} \gamma_{A,B} \right)$.
\end{defn}

\begin{restatable}[Guarantees for \textsc{r-Split}]{theorem}{theoremrsplit}
\label{thm:main-result}
    Let $\mathcal{A}$ be an algorithm for matroid constrained maximization of set functions. 
    Suppose that for any non-negative monotone function $g\colon 2^W \to \mathbb{R}$, $|W|=m$, with generalized curvature $\alpha_g$ and submodularity ratio $\gamma_g$, algorithm $\mathcal{A}$ makes $O(q(m))$ queries to the value of $g$, where $q$ is a polynomial, and returns a solution with approximation ratio $R(\alpha_g, \gamma_g)$. 
    
    If $f \colon 2^{[n]} \to \mathbb{R}$ is a non-negative monotone function with generalized curvature $\alpha$, submodularity ratio $\gamma$, and elementary submodular rank $r+1$, then \textsc{r-Split} with subroutine $\mathcal{A}$ runs in time $O(2^r n^rq(n))$ and returns a solution with approximation ratio $\max\{R(\alpha, \gamma), R(\alpha_r, 1)\}$. 
\end{restatable}
\begin{proof}[Proof of Theorem~\ref{thm:main-result}] 
    We first discuss the computational cost. If we run $\mathcal{A}$ on $f$, we get an approximation ratio of at least $R(\alpha, \gamma)$, at a cost of $O(q(n))$. 
    For \textsc{r-Split}, there are $\binom{n}{r} = O(n^r)$ sets $B\subset V$ of cardinality $r$. For each, there are $2^r$ subsets $A \subseteq B$. 
    Hence we have $2^r\binom{n}{r}$ possible $\Pi(A,B)$. On each, we run $\mathcal{A}$, which runs in $O(q(n-|A|))$, because $f_{A,B}$ can be regarded as a function on $2^{V\setminus B}$. 
    The final step is to pick the optimal value among the solutions returned for each subproblem, which can be done in $O(2^r\binom{n}{r})$ time. 
    Hence the overall cost is $O(q(n)) + 2^r\binom{n}{r} O(q(n-|A|))+O(2^r\binom{n}{r}) = O(2^r n^r q(n))$. 

    Now, we discuss the approximation ratio. The solution returned for $f_{A,B}$ has approximation ratio $R(\alpha_{A,B},\gamma_{A,B})$, by the assumed properties of $\mathcal{A}$. 
    Since $f$ has elementary submodular rank $r+1$, there exist submodular $f_0$ and elementary submodular $f_{i_1}, \ldots, f_{i_{r}}$ such that $f = f_0 + f_{i_1} + \cdots + f_{i_{r}}$. 
    Let $B_f = \{i_1, \ldots, i_{r}\}$. Then, for this set $B_f$ and any $A \subseteq B_f$, we know that $f_{A,B_f}$ is submodular, by Proposition~\ref{prop:submodular-restriction}, and hence $\gamma_{A,B_f} = 1$. 
    Picking the optimum value among the solutions returned for the subproblems involving $B_f$ ensures an overall approximation ratio with $\min_{A\subseteq B_f}\gamma(A,B_f) = 1$, that is, $R(\alpha_r,1)$. 
    In summary, we are guaranteed to obtain a final solution with approximation ratio $\max \{ R(\alpha,\gamma), R(\alpha_r,1) \}$. 
\end{proof}

\begin{rem}[Approximation ratio] 
    The approximation ratio $R(\alpha, \gamma)$ usually improves as $\gamma$ increases, as seen 
    for example in the bound $\frac{1}{\alpha}(1-e^{-\alpha \gamma})$ in Theorem \ref{thm:bian}.
\end{rem} 

\begin{rem}[Time complexity] 
    For fixed $r$, \textsc{r-Split} provides 
    a polynomial time approximation 
    for elementary submodular rank-$(r+1)$ functions. 
    If we further assume knowledge of the cones involved in a decomposition of $f$, then we need only optimize over one split, and we can drop the time complexity factor $n^r$.  
    In this case we get a running time of $O(2^r \cdot poly(n))$. In general, we should not expect to be able to do better than this; see Theorem 
    \ref{thm:lower}. 
    
    With the extra information about the decomposition, this is a fixed parameter tractable (FPT) time $O(1)$ approximation for positive monotone function maximization with matroid constraints, parameterized by the elementary submodular rank. This suggests that determining the cones in the decomposition may be a difficult problem.  
\end{rem} 

{While in general computing the elementary submodular rank is difficult, structural information about the function can provide insights into the rank.} 

\begin{example} \label{eg:clique} 
    {The $k$-clique problem asks whether a given graph has a clique of size $k$. This can be modeled as a constrained maximization problem as follows. Given $G = (V,E)$, let $f(S)$ be the number of edges in the induced subgraph with vertex set $S$. Thus, we want to maximize $f$ subject to the constraint that $|S| \le k$. If the maximum value is $\binom{k}{2}$, then there exists a $k$-clique. This function is supermodular and hence has the maximum elementary submodular rank, which is $|V|$. }
\end{example}

\begin{example} \label{eg:weighted-maxcut}
    {Consider a weighted graph $G = (V,E,w)$, 
    where $w : E \to \mathbb{R}$ is allowed to take negative values.  
    Consider the weighted max cut problem, which is to maximize  
    \[
        f(S) = \sum_{u \in S, v \in V \setminus S} w(u,v).
    \]
    This function is neither submodular nor supermodular, 
    due to the negative and positive edge weights. We define
    \[
        f_+(S) = \sum_{\substack{u \in S, v \in V \setminus S \\ w(u,v) \ge 0}} w(u,v) 
    \qquad \text{and} \qquad 
        f_-(S) = \sum_{\substack{u \in S, v \in V \setminus S \\ w(u,v) < 0}} w(u,v) . 
    \]
    Then $f(S)=f_+(S)+f_-(S)$.
    The function $f_+$ is submodular. It remains to find the elementary submodular rank of $f_-(S)$, which is supermodular on $2^{V}$. Let $G_- = (V_-, E_-)$ be the graph of the negative edges. 
    Then $f_-$ restricted to $2^{V_-}$ is also in the interior of the supermodular cone and hence has rank $|V_-|$, by Theorem~\ref{prop:maximum-elementary-rank}. Since $f_-$ is zero outside of $2^{V_-}$, the rank of $f_-$ on $2^V$ is still $|V_-|$.
    Thus, the elementary submodular rank of $f$ is at most $|V_-|+1$.  
}
\end{example}

\begin{restatable}[Lower bound]{theorem}{lowerbound} \label{thm:lower} 
    Any deterministic procedure that achieves an $O(1)$ approximation ratio for maximizing elementary submodular rank-$(r+1)$ set functions requires at least $2^r$ function queries. 
\end{restatable}
\begin{proof}
    Assume for contradiction that such a procedure $\mathcal{A}$ exists. Let $f$ be a set function with elementary rank-$(r+1)$ and let $f_1, \ldots, f_{2^r}$ be the submodular splits of the functions. Since the procedure samples fewer than $2^r$ values, there is a piece $f_i$ that is not queried. Consider a modified function $\hat{f}$, with the same 
    pieces, 
    i.e., $\hat{f}_j = f_j$  except that $\hat{f}_i = f_i + c$. Since $c$ can be arbitrary and the procedure is deterministic, it will not query $\hat{f}_i$. Hence will not return an $O(1)$ approximation, which gives the desired contradiction. 
\end{proof}

Let us now instantiate corollaries of Theorem~\ref{thm:main-result}. 
We fix the elementary rank to be $r+1$. The runtime will be  exponential in $r$ but polynomial in $n$. 

\begin{cor} 
\label{cor:non-oblivious} 
If $f$ is a non-negative, monotone function with elementary submodular rank $r+1$, then for the cardinality constrained problem, 
for all $\epsilon > 0$, 
\textsc{r-Split} with \textsc{Non-Oblivious Local Search Greedy} as the subroutine 
has an approximation ratio of $\max(R(\alpha, \gamma)$, $(1-\hat{\alpha} e^{-1}) - O(\epsilon))$ and runs in $O(\epsilon^{-1} 2^rn^r \cdot poly(n))$ time. 
\end{cor}
\begin{proof}
    Use \textsc{Non-Oblivious Local Search Greedy} and Theorem \ref{thm:svidernko}.
\end{proof}

\begin{cor} 
\label{cor:cardinality} 
If $f$ is a non-negative, monotone function with submodularity ratio $\gamma$, curvature $\alpha$, and elementary submodular rank $r+1$, then for the cardinality constrained problem, \textsc{r-Split} with \textsc{Greedy} as the subroutine has an approximation ratio of $\alpha_r^{-1}(1-e^{-\alpha_r})$ and runs in $O(2^rn^r \cdot nm)$ time.  
\end{cor}

\begin{cor} 
\label{cor:matroid} 
If $f$ is a non-negative, monotone function with elementary submodular rank $r+1$, then for the matroid constrained problem, 
\textsc{r-Split} with the algorithm from \cite{Clinescu2011MaximizingAM} as the subroutine
has an approximation ratio of $(1-e^{-1})$ and runs in $\tilde{O}(2^rn^r \cdot n^8)$ time. 
\end{cor}
\begin{proof}
    Use the algorithm from \cite{Clinescu2011MaximizingAM} and Theorem \ref{thm:matorid-sub} for the guarantee. It remains to show that the procedure from \cite{Clinescu2011MaximizingAM} terminates in $\tilde{O}(n^8)$ time, even if the input is not submodular. The algorithm from \cite{Clinescu2011MaximizingAM} consists of two steps. First, is the \textsc{Continuous Greedy} algorithm. Second, is \textsc{Pipage Round}~\cite{Ageev2004PipageRA}. From \cite{Clinescu2011MaximizingAM} we have that \textsc{Continuous Greedy} terminates after a fixed number of steps. This would be true even if the input function $f$ is not submodular and always returns a point in the base polytope of the matroid. Second, using \cite[Lemma 3.5]{Clinescu2011MaximizingAM}, we see that \textsc{Pipage Round} terminates in $\tilde{O}(n^8)$. 
    for any point in the base matroid of the polytope. 
\end{proof}

In Table~\ref{tab:max-compare} we summarize the results from Corollaries \ref{cor:non-oblivious}, \ref{cor:cardinality}, and~\ref{cor:matroid} 
and how they compare with previous results given above in Theorems~\ref{thm:matorid-sub}, \ref{thm:svidernko}, \ref{thm:bian}, \ref{thm:gat1}, and \ref{thm:gat2}. Table~\ref{tab:max-compare} should be read as follows. 
The first part of the table shows approximation ratio guarantees from prior work. 
The second part shows the guarantees we obtain for extensions of those results to the case of functions with low elementary rank. 
For example, the first row in the table shows that for submodular functions there exists an algorithm for the matroid constrained problem with running time $\tilde{O}(n^8)$ and approximation ratio guarantee $1-e^{-1}$. 
Our results extend this to functions with elementary submodular rank $r+1$, for which we pay a time penalty of $O(2^rn^r)$, as shown in the last line of the table. 

\begin{table}[!ht]
    \caption{Approximation ratios and time complexity for maximizing a monotone, normalized function $f$ on $2^{[n]}$ of total curvature $\hat{\alpha}$, generalized curvature $\alpha$, submodularity ratio $\gamma$, and elementary submodular rank $r+1$. 
    Here, $m$ is the bound in the cardinality constraint, and $\rho$ is the matroid rank in the matroid constraint. Results with $^*$ are in expectation over randomization in the algorithm. 
    The check marks indicate requirements on $f$ or the problem setting.
    }
    \label{tab:max-compare}
    \centering
    \setlength{\tabcolsep}{3pt}
    \begin{tabular}{p{0.1in}ccccccc}
    \toprule 
      & \makecell{Sub-\\modular} & \makecell{Low \\Rank} & \makecell{Card. \\Constr.} & \makecell{Matroid \\Constr.} & \makecell{Approximation \\Ratio}  & Time & Ref.\\
      \midrule
        \multirow{7}{*}{\rotatebox{90}{\;\;\;\;\;\;Prior Work}}
        & \checkmark & - & - & \checkmark & $1-e^{-1}$ $^*$ & $\tilde{O}(n^8)$ & Thm~\ref{thm:matorid-sub} 
        \cite{Clinescu2011MaximizingAM}\\ 
        & \checkmark & - & \checkmark & - & $1-\hat{\alpha}e^{-1} - O(\epsilon)$ & $O(\epsilon^{-1} \cdot poly(n))$ & Thm~\ref{thm:svidernko} 
        \cite{sviridenko2017optimal} \\
        & - & - & \checkmark & - & $\frac{1}{\alpha}(1-e^{-\alpha \gamma})$ & $O(nm)$ & Thm~\ref{thm:bian} 
        \cite{bian2017guarantees}\\
        & - & - & - & \checkmark & $\frac{0.4 \gamma^2}{\sqrt{\rho\gamma} + 1}$ & $O(n^2)$ & Thm~\ref{thm:gat1} 
        \cite{gatmiry2018non}\\
        & - & - & - & \checkmark & $(1+(1-\alpha)^{-1})^{-1}$ & $O(n^2)$ & Thm~\ref{thm:gat2}
        \cite{gatmiry2018non}\\ \midrule \\[-.3cm]
        \multirow{3}{*}{\rotatebox{90}{This Work\!\!}} & - & \checkmark & \checkmark & - & $1-\hat{\alpha}e^{-1}- O(\epsilon)$ & $O(\epsilon^{-1} 2^rn^r \cdot poly(n))$  & Cor~\ref{cor:non-oblivious}\\[.1cm]
        & - & \checkmark & \checkmark & - & $\frac{1}{\alpha_r}(1-e^{-\alpha_r})$ & $O(2^r n^r \cdot nm)$ & Cor~\ref{cor:cardinality} \\[.1cm]
        & - & \checkmark & - & \checkmark & $1-e^{-1}$ $^*$  & $\tilde{O}(2^rn^r \cdot n^8)$ & Cor~\ref{cor:matroid}  \\
        \bottomrule 
    \end{tabular}
\end{table}

\subsection{Minimization of Ratios of Set Functions} 
\label{app:previous-works-ratio}

\begin{algorithm}[H]
\caption{\textsc{Ratio Greedy}}
\label{alg:ratiogreedy}
\small 
\begin{algorithmic}[1]
\Function{Ratio Greedy}{$f$,$g$}
    \State Initialize: $S_0 = \emptyset$, $R = V$ 
    \While{$R \neq \emptyset$}
        \State $u = \argmin_{v \in R} \frac{f(\{v\}\cup S_i)}{g(\{v\} \cup S_i)}$
        \State $S_{i+1} = S_i \cup \{u\}$
        \State $R = \{ v \in R : g(\{v\} \cup S_{i+1}) > 0\}$
    \EndWhile
    \State \Return $S_i$
\EndFunction
\end{algorithmic}
\end{algorithm}

The second application of our decomposition is to minimize ratios of set functions. We first review definitions and previous results. 
\begin{problem}
\label{prob:ratio}
Given set functions $f$ and $g$, solve $\min_{S} \frac{f(S)}{g(S)}$. 
\end{problem}
For this problem, we can use a version of \textsc{Greedy} known as \textsc{Ratio Greedy}, see Algorithm~\ref{alg:ratiogreedy}. \citet{pmlr-v48-baib16} obtained following guarantees for \textsc{Ratio Greedy} for different submodular or modular combinations of $f$ and $g$. 
\begin{theorem}[\citet{pmlr-v48-baib16}] 
\label{thm:29} 
For the ratio of set function minimization problem: 
\begin{enumerate}
    \item \textsc{Greedy Ratio} finds the optimal solution if $f,g$ are modular. 
    \item \textsc{Greedy Ratio} finds a $1-e^{-1}$ approximation if $f$ is modular and $g$ is submodular. 
    \item \textsc{Greedy Ratio} finds a $1/(1-e^{\hat{\alpha}^f-1})$ approximation if $f$ and $g$ are submodular. Here $\hat{\alpha}^f$ is the total curvature of $f$. 
\end{enumerate}
\end{theorem}
To quantify the approximation ratio for general monotone $f,g$, we need a few definitions.   

\begin{defn}[\citet{pmlr-v84-bogunovic18a}] 
\label{def:inverse-curvature}
\begin{itemize}
    \item The \emph{generalized inverse curvature} of a non-negative set function $f$ is the smallest $\tilde{\alpha}^f$ such that for all $T, S \in 2^V$ and for all $e \in S \setminus T$, 
    \[
        \Delta(e | S \setminus \{e\}) \ge (1-\tilde{\alpha}^f) \Delta(e | (S \setminus \{e\}) \cup T). 
    \]
    \item The \emph{curvature} of $f$ with respect to $X$ is $\hat{c}^f(X) := 1 - \frac{\sum_{e \in X} (f(X) - f(X\setminus \{e\})}{\sum_{e\in X} f(\{e\})}$. 
\end{itemize}
\end{defn}
It is known that $f$ is submodular if and only if $\tilde{\alpha}^f = 0$. 
Using these notions, \cite{10.5555/3172077.3172251} obtains an approximation guarantee for minimization of $f/g$ for monotone submodular $f$ and monotone~$g$, and \cite{Wang2019MinimizingRO} obtains a guarantee when both $f$ and $g$ are monotone functions. 

\begin{theorem}[\citet{10.5555/3172077.3172251}] 
\label{thm:33} 
For minimizing $f/g$, where $f$ is a normalized, positive, monotone submodular function and $g$ is a normalized, positive, monotone function, \textsc{Greedy Ratio} finds a subset $X \subseteq V$ with 
\[
    \frac{f(X)}{g(X)} \le \frac{1}{\gamma_{\emptyset,|X^*|}^g}\frac{|X^*|}{1+(|X^*|-1)(1-\hat{c}^f(X^*))} \frac{f(X^*)}{g(X^*)},
\]
where $X^*$ is the optimal solution. Here $\gamma^g$ refers to the submodularity ratio of $g$. 
\end{theorem}

\begin{theorem}[\citet{Wang2019MinimizingRO}] 
\label{thm:34} 
For minimizing $f/g$, where $f$ and $g$ are normalized, positive, monotone set functions, \textsc{Greedy Ratio} outputs a subset $X \subseteq V$ with 
\[
    \frac{f(X)}{g(X)} \le \frac{1}{\gamma^g_{\emptyset,|X^*|}}\frac{|X^*|}{1+(|X^*|-1)(1-\alpha^f)(1-\tilde{\alpha}^f)} \frac{f(X^*)}{g(X^*)},
\]
where $X^*$ is the optimal solution. Here $\gamma^g$ refers to the submodularity ratio of $g$, and $\alpha^f,\tilde{\alpha}^f$ are the generalized curvature and the generalized inverse curvature of $f$.
\end{theorem}

\subsection{Functions with Bounded Elementary Rank} 

\begin{algorithm}[H]
\caption{\textsc{r-Split Ratio}}
\label{alg:modified-ratio-1}
\small 
\begin{algorithmic}[1]
\Function{r-Split Ratio}{$f,g$, $r$, $\mathcal{A}$ - \textsc{Greedy Ratio}}
    \For{$A \subseteq B \subseteq V$ with $|B| = r$}
        \State run $\mathcal{A}$ on $f_{A,B}$ and $g_{A,B}$
    \EndFor
    \State run $\mathcal{A}$ on $f,g$\;
    \State \Return Best seen set  
\EndFunction
\end{algorithmic}
\end{algorithm}

For minimization of ratios of set functions we propose 
\textsc{r-Split Ratio} and \textsc{r-Double Split Ratio} with \textsc{Greedy Ratio} subroutine, shown in Algorithms~\ref{alg:modified-ratio-1} and ~\ref{alg:modified-ratio-2}.  
We formulate guarantees for these algorithms. 
These improve previous guarantees if the functions have low elementary submodular rank. Here we have many options how to split. First we split only $f$, then we split both $f$ and $g$. 

\begin{restatable}[Guarantees for \textsc{r-Split Greedy Ratio}]{theorem}{rsplitgreedyratio} 
\label{thm:main-result2a}
For the minimization of $f/g$ where $f, g$ are normalized, positive, monotone functions, assume $f$ has elementary submodular rank $r+1$. Let $X^*$ be the optimal solution.
Then \textsc{r-Split Ratio} with \textsc{Greedy Ratio} subroutine at a time complexity of $O(2^r n^r \cdot n^2)$, has approximation ratio 
\[
    \frac{1}{\gamma_{\emptyset,|X^*|}^g}\frac{|X^*|}{1+(|X^*|-1)(1-\hat{c}^f(X^*))}.
\] 
If, in addition, $g$ is submodular then the approximation ratio is $1/(1-e^{\hat{\alpha}^f-1})$. 
\end{restatable}
\begin{proof}[Proof of Theorem~\ref{thm:main-result2a}]
The statement follows from analogous arguments to those in the proof of Theorem~\ref{thm:main-result}. 
There is a set $B_f$ such that $f_{A,B_f}$ is submodular on all $A \subseteq B_f$. We use Theorem~\ref{thm:33} to get the approximation ratio. If $g$ is submodular, then its restrictions are submodular, and we minimize the ratio of two submodular functions. Hence can use Theorem~\ref{thm:29} to obtain the approximation ratio. 
\end{proof}
The first statement provides the same guarantee as a result of \cite{10.5555/3172077.3172251} but for a more general class of functions. Our result can be interpreted as grading the conditions of \cite{Wang2019MinimizingRO} to obtain similar guarantees as the more restrictive results of \cite{10.5555/3172077.3172251}. 

\begin{algorithm}[!ht]
\caption{\textsc{r Double Split Greedy Ratio}}
\label{alg:modified-ratio-2}
\small 
\begin{algorithmic}[1]
\Function{r Split Greedy Ratio}{$f,g$ - set functions, $r_f, r_g$ - integer, $\mathcal{A}$ - \textsc{Greedy Ratio}}
    \For{$A \subseteq B \subseteq V$ with $|B| = r_f+r_g$}
        \State run $\mathcal{A}$ on $f$ and $g$ restricted to $\Pi(A,B)$\
    \EndFor
    \State run $\mathcal{A}$ on $f,g$\;
    \State \Return Best seen set  
\EndFunction
\end{algorithmic}
\end{algorithm}

\begin{restatable}[Guarantees for \textsc{r Double Split Greedy Ratio}]{theorem}{rsplitgreedyratiotwo} 
\label{thm:main-result2} 
Assume $f$ and $g$ are normalized, positive, monotone functions, 
with elementary submodular ranks $r_f+1$ and $r_g+1$. 
For the minimization of $f/g$, 
\textsc{r Double Split Ratio} with \textsc{Greedy Ratio} subroutine at a time complexity of $O(2^{r_f+r_g} n^{r_f+r_g} \cdot n^2)$, has approximation ratio 
$ \frac{1}{1-e^{\hat{\alpha}^f-1}},  $
\end{restatable}
\begin{proof}[Proof of Theorem~\ref{thm:main-result2}]
Use $r = r_f+r_g$ in Algorithm \ref{alg:modified-ratio-2}. Then there is a $B_f$ such that $f_{A,B_f}$ is submodular and a $B_g$ such that $g_{A,B_g}$ is submodular. Let $B_{f,g} = B_f \cup B_g$. Then $|B_{f,g}| \le |B_f| + |B_g| = r_f + r_g$ and both $f_{A,B_{f,g}}$ and $g_{A,B_{f,g}}$ are submodular. We minimize a ratio of submodular functions, and use Theorem \ref{thm:29} for the approximation ratio. 
\end{proof}

This improves the approximation ratio guarantee of \cite{Wang2019MinimizingRO} to that of \cite{pmlr-v48-baib16}, which was valid only if $f,g$ were both submodular. 
Table~\ref{tab:max-compare-ratio} compares our results with prior work. 

\begin{table}[!ht]
    \caption{Approximation ratios and time complexity for minimizing $f/g$ for monotone, normalized functions $f$ and $g$ on $2^{[n]}$. 
    Here $f$ has total curvature $\hat{\alpha}^f$, generalized curvature $\alpha^f$, generalized inverse curvature $\tilde{\alpha}^f$, curvature $\hat{c}^f$, and elementary submodular rank $r_f+1$. $g$ has submodularity ratio $\gamma^g$ and elementary submodular rank $r_g+1$. $X^*$ is the optimal solution.}
    \label{tab:max-compare-ratio}
    \centering
    \small 
    \renewcommand{\arraystretch}{0}
    \setlength{\tabcolsep}{5.5pt}
    \begin{tabular}{p{.01in}C{.74in}C{.80in}C{1.73in}C{1.1in}C{0.85in}}
    \toprule 
      & Numerator & Denominator & Approx.\ Ratio  & Time & Ref.\\ \midrule
        \multirow{5}{*}{\rotatebox{90}{Prior Work\;\;\;\;}}  & Modular & Modular & 1 & $O(n^2)$ & Thm~\ref{thm:29} 
        \cite{pmlr-v48-baib16} \\
        & Modular & Submodular &  $1-e^{-1}$ & $O(n^2)$ & Thm~\ref{thm:29} 
        \cite{pmlr-v48-baib16} \\ 
        & Submodular & Submodular & $\frac{1}{1-e^{\hat{\alpha}^f-1}}$ & $O(n^2)$ & Thm~\ref{thm:29} 
        \cite{pmlr-v48-baib16} \\
        & Submodular & - & $\frac{1}{\gamma_{\emptyset,|X^*|}^g}\frac{|X^*|}{1+(|X^*|-1)(1-\hat{c}^f(X^*))}$ & $O(n^2)$ & Thm~\ref{thm:33} 
        \cite{10.5555/3172077.3172251}\\
        & - & - & $\frac{1}{\gamma^g_{\emptyset,|X^*|}}\frac{|X^*|}{1+(|X^*|-1)(1-\alpha^f)(1-\tilde{\alpha}^f)}$ & $O(n^2)$ & Thm~\ref{thm:34} 
        \cite{Wang2019MinimizingRO}\\\midrule 
        \multirow{3}{*}{\rotatebox{90}{\makecell{This Work}}} & Low Rank & - & $\frac{1}{\gamma_{\emptyset,|X^*|}^g}\frac{|X^*|}{1+(|X^*|-1)(1-\hat{c}^f(X^*))}$  & $O(2^{r_f}n^{r_f} \cdot n^2)$ & Thm~\ref{thm:main-result2a}\\
        & Low Rank & Submodular & $\frac{1}{1-e^{\hat{\alpha}^f-1}}$ & $O(2^{r_f} n^{r_f} \cdot n^2)$ & Thm~\ref{thm:main-result2a} \\
        & Low Rank & Low Rank & $\frac{1}{1-e^{\hat{\alpha}^f-1}}$  & \!\!\!\!$O(2^{r_f+r_g}n^{r_f+r_g} \cdot n^2)$ & Thm~\ref{thm:main-result2}  \\
        \bottomrule 
    \end{tabular}
\end{table}

\section{Experiments} 
\label{sec:experimental-results}

Our theory shows that our algorithms extend theoretical guarantees for submodular function optimization to low elementary submodular rank optimization. 
In this section we take set functions that are not submodular, run our new algorithms, and show an improved performance. 
We consider the following functions. \footnote{Computer code and all experimental details for our algorithms and experiments are provided in \href{https://github.com/rsonthal/Submodular-Set-Function-Optimization}{GitHub repository}.} 

\paragraph{Determinantal functions} 
Let $\Sigma = XX^T$, where $X \in \mathbb{R}^{n \times d}$ is a Gaussian random matrix (i.e., entries are i.i.d.\ samples from a standard Gaussian distribution). To ensure $\Sigma$ is positive definite, we impose $d \ge n$. 
Given $S \subseteq [n]$, denote by $\Sigma_S$ the $|S| \times |S|$ matrix indexed by the elements in $S$. Given $\sigma \in \mathbb{R}$, we define 
{$f(S) := \log\det(I + \sigma^{-2}\Sigma_S)$ for $S\subseteq [n]$}. 
Functions of this type are supermodular \cite{bian2017guarantees} and they appear in determinantal point processes \cite{10.5555/2481023}. 

\paragraph{Bayesian A-optimality functions} 
The Bayesian A-optimality criterion in experimental design seeks to minimize the variance of a posterior distribution as a function of the set of observations. Let $x_1, \ldots, x_n \in \mathbb{R}^d$ be $n$ data points. For any $S \subseteq [n]$, let $X_S\in\mathbb{R}^{n\times|S|}$ be the matrix collecting the data points with index in $S$. Let $\theta \sim \mathcal{N}(0, \beta^{-2}I)$ be a parameter and let $y_S = \theta^TX_S + \xi$, where $\xi \sim \mathcal{N}(0,\sigma^2 I)$. 
Let $\Sigma_{\theta | y_S}$ be the posterior covariance of $\theta$ given $y_S$. 
Then we define $f(S) = \Tr(\beta^{-2} I) - \Tr(\Sigma_{\theta|y_S}) = \frac{d}{\beta^2} - \frac{1}{\beta^2}\Tr((I + (\beta \sigma)^{-2} X_SX_S^T)^{-1}).$
Maximizing $f$ identifies a set of observations that minimizes the variance of the posterior. The work \cite{bian2017guarantees} provides bounds on $\alpha$ and $\gamma$ for this function. 

\paragraph{Column subset selection} 
Given a matrix $A$, we ask for a subset $S$ of the columns that minimizes 
$f(S) = \|A\|_F^2 - \|A_S A_S^\dag A\|_F^2 .$
Here, $A_S^\dag$ is the Moore-Penrose pseudoinverse. Hence $A_SA_S^\dag$ is the orthogonal projection matrix onto the column space of $A_S$. 

\paragraph{Random functions}
We take a uniform random sample from $[0,1]$ of size $2^n$ and sort it in increasing order as a list $L$. We then construct a monotone function $f$ by assigning to $f(\emptyset)$ the smallest value in $L$ and then, for $M=1,\ldots, n$, assigning to $f(S)$, $S\subseteq [n]$, $|S|=M$, in any order, the next $\binom{n}{M}$ smallest elements of $L$.

\paragraph{Submodularity ratio and generalized curvature} 
We compute $\alpha_r$ and $\gamma_r$ and the resulting approximation ratio guarantee for constrained maximization for functions of the four different types defined above, for $n = 8$ and log number of pieces $0\leq r\leq 4$. We report the mean and standard error for $50$ functions in each case. 
Figure~\ref{fig:boundk} shows that the approximation ratio guarantee can increase quickly: for \textsc{Determinantal} it improves by 400\% by $r=4$. 

\begin{figure}[!ht]
    \centering
    \subfloat[$\frac{1}{\alpha_r}(1-e^{-\alpha_r \gamma_r})$\label{fig:boundk}]{\includegraphics[width=0.33\linewidth]{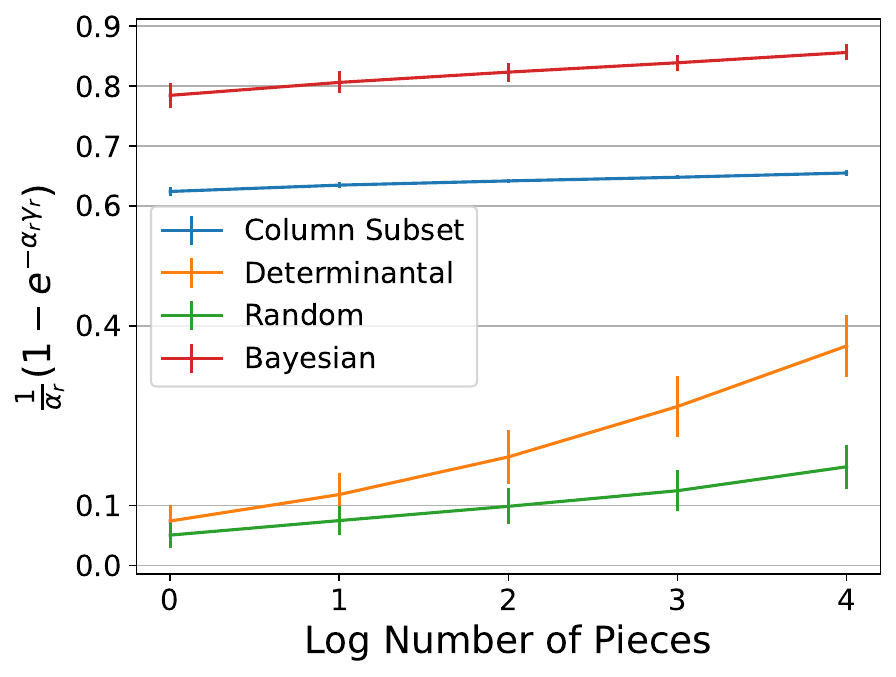}}\hfill
    \subfloat[Relative errors \label{fig:error}]{\includegraphics[width=0.33\linewidth]{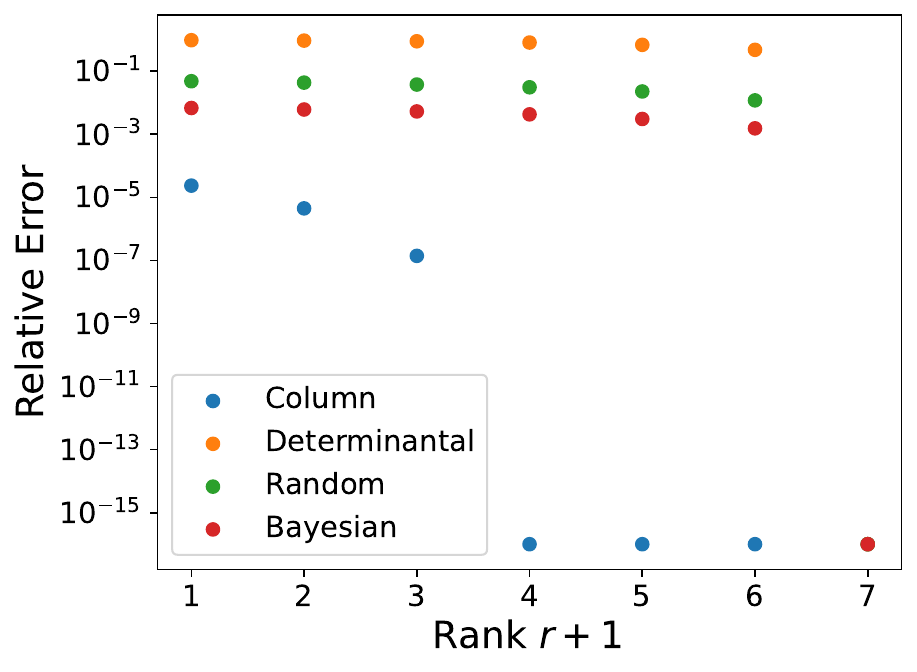}}\hfill
    \subfloat[Running time in seconds \label{fig:times}]{\includegraphics[width=0.33\linewidth]{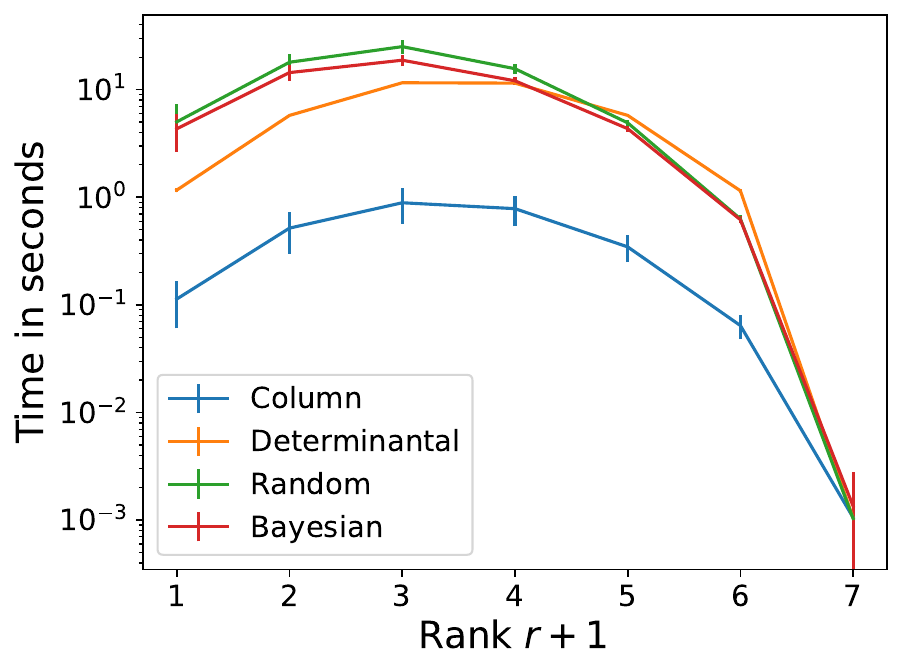}}
    \caption{Shown are (a) how the bound from \cite{bian2017guarantees} changes when we use $\alpha_r$ and $\gamma_r$ (see Definition \ref{def:alpha-gamma-min-max}); higher values correspond to better guarantees, (b) the relative error when approximating a function by an elementary rank-$(r+1)$ function, and (c) the running times for computing the approximation. }
    \label{fig:alpha-gamma-k}
\end{figure}

\paragraph{Low elementary rank approximations} Given $f\colon\{0,1\}^n \to \mathbb{R}$ and an integer $r$, we seek an \emph{elementary submodular rank-$r$ approximation} of $f$. This is a function $g\colon \{0,1\}^n \to \mathbb{R}$ that minimizes $\|f-g\|_{\ell_2}$, with $g$ having elementary submodular rank $r$. 
The set of elementary submodular rank $r$ functions is a union of convex cones, which is, in general, not convex.
However, for fixed $\pi^{(1)}, \ldots, \pi^{(r)}$, finding the closest point to $f$ in $\mathcal{L}_{\pi^{(1)}} + \cdots + \mathcal{L}_{\pi^{(r)}}$ is a convex problem. To find $g$, we compute the approximation for all rank $r$ convex cones and pick the function with the least error. Computing the projection onto each cone may be challenging due to the high number of facet-defining inequalities. We use \textsc{Project and Forget} \citep{JMLR:v23:20-1424}. 

We compute low elementary submodular rank approximations for $n=7$ and $1\leq r+1 \leq 7$. Figure \ref{fig:error} shows the relative error ${\|f-g\|_{\ell_2}}/{\|f\|_{\ell_2}}$ and Figure \ref{fig:times} the running times. We see that \textsc{Column} has a low elementary submodular rank. The computation time peaks near $r+1=n/2$ and decreases for larger $r+1$ as there are fewer Minkowski sums and constraints.

\paragraph{\textsc{r-Split Greedy} with small $n$} 
We evaluate the improvement that \textsc{r-Split} provides over \textsc{Greedy}. Let $n = 20$ and maximize set functions subject to a cardinality constraint $m=10$. Figure~\ref{fig:ratio} shows the approximation ratio for \textsc{Greedy} and $\textsc{r-Split Greedy}$ for $r = 1,2,3$, as well as how often the optimal solution was found. All four algorithms outperform their theoretical bounds. In all cases, increments in $r$ increase the percentage of times the (exact) optimal solution is found. 

\begin{figure}[!ht]
    \centering 
    \subfloat[Approximation Ratio \label{fig:ratio}]{\includegraphics[width = 
    0.33\linewidth]{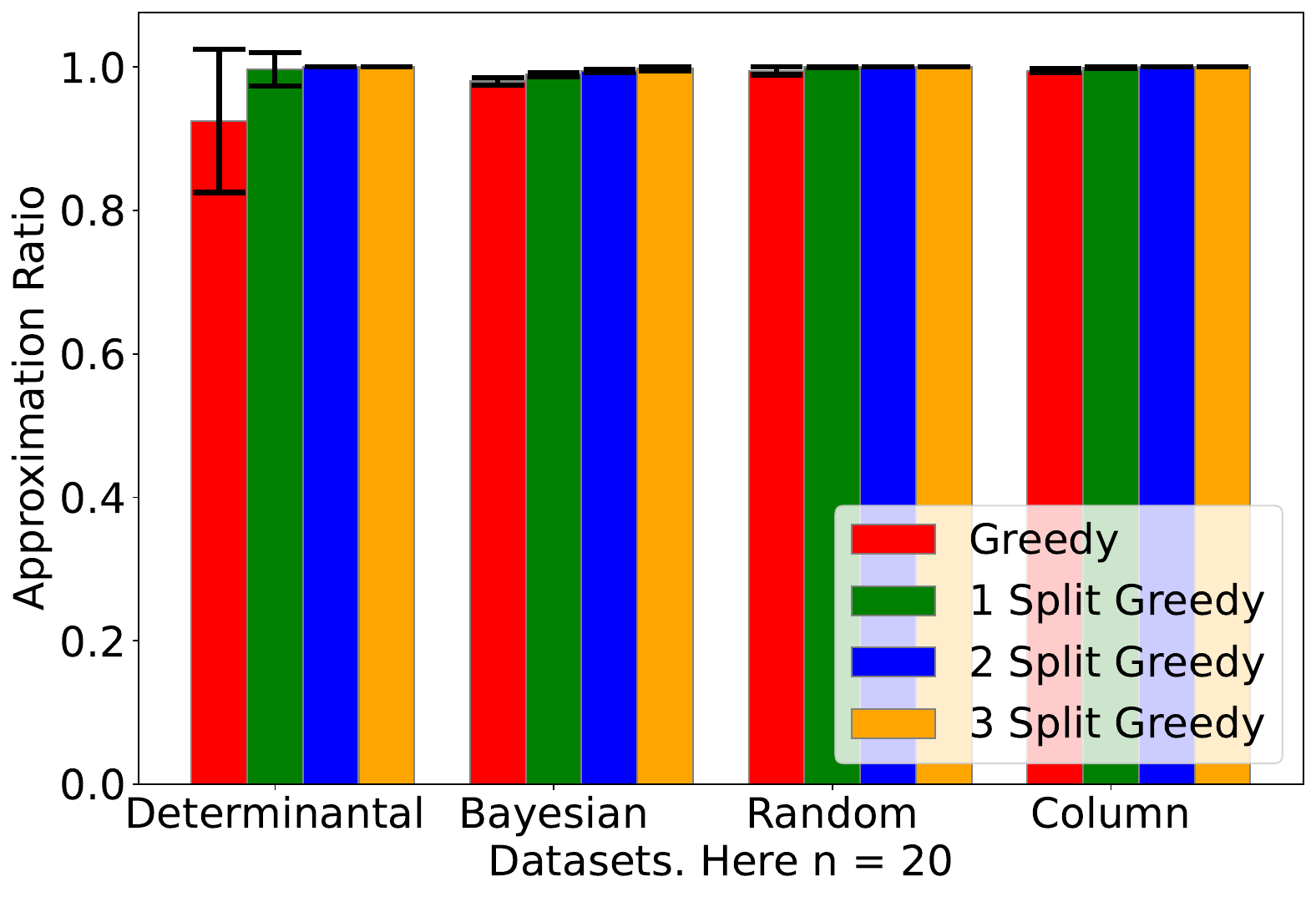}}\hfill
    \subfloat[Exact Solutions \label{fig:opt}]{\includegraphics[width = 0.33\linewidth]{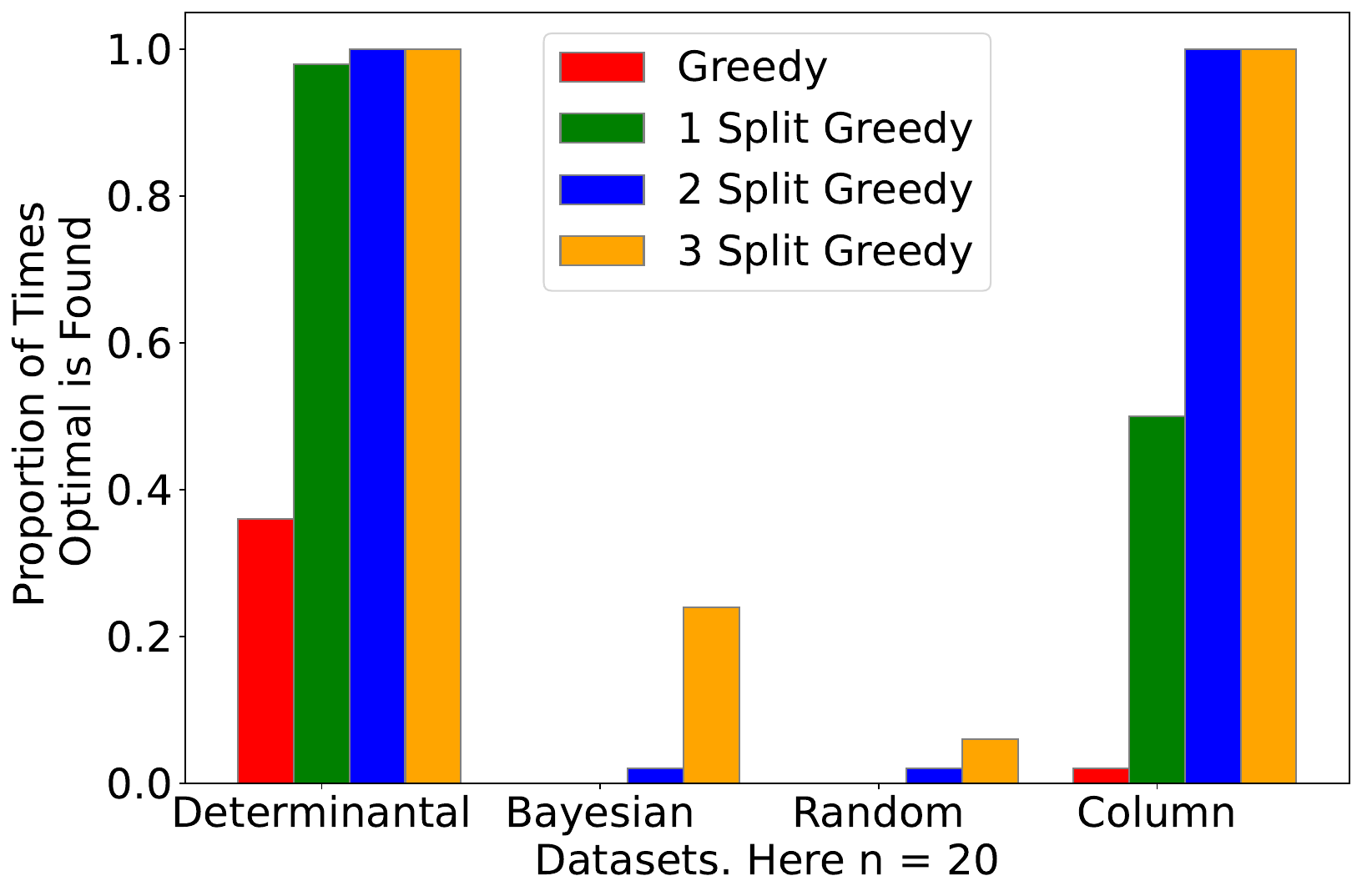}}
    \subfloat[\label{fig:ratios} Split / Greedy Ratio]{\includegraphics[width=0.33\linewidth]{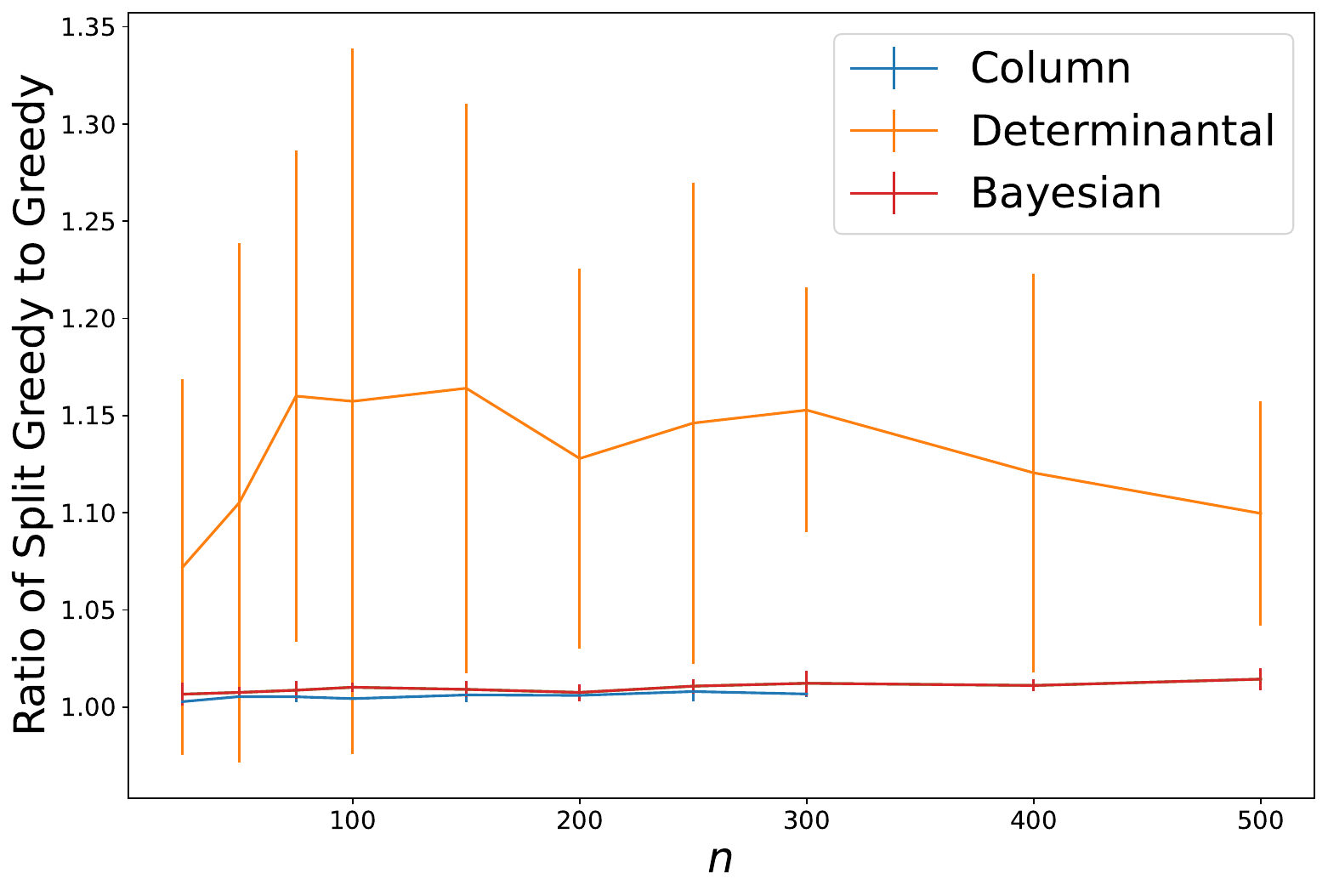}}
    \caption{Shown are (a) the approximation ratio for the solution found versus the optimal solution averaged over 50 objective functions in each type, (b) the proportion of instances for which the algorithms find the optimal solution, (c) the ratio of the solution found by \textsc{1 Split} and \textsc{Greedy}.} 
    \label{fig:bar}
\end{figure}

\paragraph{\textsc{r-Split Greedy} with large $n$}
We now consider larger values of $n$, and a cardinality of $m=15$. Since the exact optimal solution is unknown, Figure~\ref{fig:ratios} plots the ratio of \textsc{r-Split Greedy} to \textsc{Greedy}. For \textsc{Bayesian} our approach improves the quality of the returned solution by about 1\% and for \textsc{Determinantal} by 5--15\% on average. Running the experiment for \textsc{Random} is not viable as we cannot store it for oracle access. Computing the \textsc{Column} function takes longer than computing \textsc{Determinantal} and \textsc{Bayesian}. Hence we only ran it until $n=300$. For $n=300$, there are still as many as $\binom{300}{15} \approx 8 \times 10^{24}$ possible solutions.

\section{Conclusion} 

We introduced the notion of submodular and supermodular rank for set functions and studied their polyhedral characterizations. 
Based on this, we developed algorithms for constrained set function maximization and ratio of set functions minimization with theoretical guarantees improving several previous guarantees for these problems. 
Our algorithms do not require knowledge of the rank decomposition and show improved empirical performance on several commonly considered tasks, even for small choices of $r$. For large $n$ it becomes unfeasible to evaluate all splits for large $r$, and one could consider evaluating only a random selection. It will be interesting to further study the rank of typical functions. The theory and practice of computing rank decompositions remain open problems with interesting consequences. Another natural extension of our work would be to consider general lattices instead of binary variables. 

\subsection*{Acknowledgments} 
RS and GM have been supported in part by DFG in SPP 2298 project 464109215. 
GM has been supported by NSF CCF-2212520, NSF DMS-2145630, and BMFTR in DAAD project 57616814 (SECAI). 
AS was supported by the NSF (DMR-2011754) and by an Alfred P. Sloan research fellowship.

\nocite{sage, Gawrilow2000, assarf2014computing}
\bibliographystyle{plainnat}
\bibliography{references}

@inproceedings{NIPS2013_7bb06076,
    author = {Martens, James and Chattopadhya, Arkadev and Pitassi, Toni and Zemel, Richard},
    booktitle = {Advances in Neural Information Processing Systems},
    title = {On the {R}epresentational {E}fficiency of {R}estricted {B}oltzmann {M}achines},
    year = {2013}
}

@article{MONTUFAR2017531,
    title = {Hierarchical {M}odels as {M}arginals of {H}ierarchical {M}odels},
    journal = {International Journal of Approximate Reasoning},
    year = {2017},
    author = {Guido Mont\'{u}far and Johannes Rauh},
}

@article{doi:10.1137/16M1077489,
    author = {Mont\'{u}far, Guido and Morton, Jason},
    title = {Dimension of {M}arginals of {K}ronecker {P}roduct {M}odels},
    journal = {SIAM Journal on Applied Algebra and Geometry},
    year = {2017},
}

@article{6796877,
    author={Le Roux, Nicolas and Bengio, Yoshua},
    journal={Neural Computation}, 
    title={Representational {P}ower of {R}estricted {B}oltzmann {M}achines and {D}eep {B}elief {N}etworks}, 
    year={2008},
}

@inproceedings{NIPS2011_8e98d81f,
    author = {Mont\'{u}far, Guido and Rauh, Johannes and Ay, Nihat},
    booktitle = {Advances in Neural Information Processing Systems},
    title = {Expressive {P}ower and {A}pproximation {E}rrors of {R}estricted {B}oltzmann {M}achines},
    year = {2011}
}

@book{lauritzen1996graphical,
	author = {Lauritzen, Steffen L.},
	publisher = {Clarendon Press},
	series = {Oxford Statistical Science Series},
	title = {Graphical {M}odels},
	year = {1996}
}

@article{Hinton2002, 
    author = {Hinton, Geoffrey E.}, title = {Training {P}roducts of {E}xperts by {M}inimizing {C}ontrastive {D}ivergence}, 
    year = {2002}, 
    publisher = {MIT Press}, 
    journal = {Neural Computation}, 
}

@inproceedings{Smolensky1986,
    author = {Paul Smolensky},
    title = {Information {P}rocessing in {D}ynamical {S}ystems: {F}oundations of {H}armony {T}heory},
    booktitle = {Symposium on Parallel and Distributed Processing},
    year = {1986},
}

@article{modeposet2015,
    author = {Guido Mont\'ufar and Johannes Rauh},
    journal ={Journal of Algebraic Statistics},
    title ={Mode {P}oset {P}robability {P}olytopes},
    year = {2016}
}

@article{JMLR:v13:song12a,
    author  = {Le Song and Alex Smola and Arthur Gretton and Justin Bedo and Karsten Borgwardt},
    title   = {Feature {S}election {V}ia {D}ependence {M}aximization},
    journal = {Journal of Machine Learning Research},
    year    = {2012},
}

@inproceedings{NIPS2005_b0bef4c9,
    author = {Narasimhan, Mukund and Jojic, Nebojsa and Bilmes, Jeff A},
    booktitle = {Advances in Neural Information Processing Systems},
    title = {Q-{C}lustering},
    year = {2005}
}

@InProceedings{pmlr-v37-wei15,
    title = {Submodularity in {D}ata {S}ubset {S}election and {A}ctive {L}earning},
    author = {Wei, Kai and Iyer, Rishabh and Bilmes, Jeff},
    booktitle =  {Proceedings of the 32nd International Conference on Machine Learning},
    year = {2015},
}

@article{1316848,
    author={Boykov, Yuri and Kolmogorov, Vladimir},
    journal={IEEE Transactions on Pattern Analysis and Machine Intelligence}, 
    title={An {E}xperimental {C}omparison of {M}in-{C}ut/{M}ax-{F}low {A}lgorithms for {E}nergy {M}inimization in {V}ision}, 
    year={2004},
}

@inproceedings{pmlr-v119-halabi20a,
    title = 	 {Optimal {A}pproximation for {U}nconstrained {N}on-{S}ubmodular {M}inimization},
    author =       {Halabi, Marwa El and Jegelka, Stefanie},
    booktitle = 	 {Proceedings of the 37th International Conference on Machine Learning},
    year = 	 {2020},
}

@book{10.5555/2481023, 
    author = {Kulesza, Alex and Taskar, Ben}, 
    title = {Determinantal {P}oint {P}rocesses for {M}achine {L}earning}, 
    year = {2012}, 
    publisher = {Now Publishers Inc.}, 
}

@book{2011cones,
	author = {Takuya Kashimura and Tomonari Sei and Akimichi Takemura and Kentaro Tanaka},
	publisher = {Department of Mathematical Informatics Graduate School of Information Science and Technology the University of Tokyo},
	series = {Mathematical engineering technical reports},
	title = {Cones of {E}lementary {I}msets and {S}upermodular {F}unctions: {A} {R}eview and {S}ome {N}ew {R}esults},
	year = {2011},
}

@article{Ageev2004PipageRA,
    title={Pipage {R}ounding: {A} {N}ew {M}ethod of {C}onstructing {A}lgorithms with {P}roven {P}erformance {G}uarantee},
    author={Alexander A. Ageev and Maxim Sviridenko},
    journal={Journal of Combinatorial Optimization},
    year={2004},
}

@article{Studeny,
	author = {Milan Studen\'y},
	journal = {arXiv:1612.06599},
	title = {Basic {F}acts {C}oncerning {E}xtreme {S}upermodular {F}unctions},
	Year = {2016}}

@article{qi2016semialgebraic,
    author = {Qi, Yang and Comon, Pierre and Lim, Lek-Heng},
    title = {Semialgebraic {G}eometry of {N}onnegative {T}ensor {R}ank},
    journal = {SIAM Journal on Matrix Analysis and Applications},
    year = {2016},
}

@misc{studeny2001, 
    title={On {M}athematical {D}escription of {P}robabilistic {C}onditional {I}ndependence {S}tructures}, 
    author = {Milan Studen\'y}, 
    year={2001},
    url={http://ftp.utia.cas.cz/pub/staff/studeny/ms-doktor-01-pdf.pdf}
}

@book{studeny2010probabilistic, 
    author = {Studen\'y, Milan}, 
    title = {Probabilistic {C}onditional {I}ndependence {S}tructures}, 
    year = {2010}, 
    publisher = {Springer Publishing Company, Incorporated}, 
}

@article{matus1999, 
    title={Conditional {I}ndependences {A}mong {F}our {R}andom {V}ariables {III}: {F}inal {C}onclusion}, 
    journal={Combinatorics, Probability and Computing}, 
    author={Mat\'u\v{s}, Franti\v{s}ek}, 
    year={1999}, 
}

@article{garcia2005algebraic,
    title = {Algebraic {G}eometry of {B}ayesian {N}etworks},
    journal = {Journal of Symbolic Computation},
    year = {2005},
    author = {Luis David Garcia and Michael Stillman and Bernd Sturmfels},
}

@article{evans2018margins,
    author = {Robin J. Evans},
    title = {{Margins of Discrete Bayesian Networks}},
    journal = {Institute of Mathematical Statistics},
    year = {2018},
}

@book{Zwiernik2015SemialgebraicSA,
  title={Semialgebraic {S}tatistics and {L}atent {T}ree {M}odels},
  author={Piotr Zwiernik},
  year={2015}, 
  publisher ={Chapman and Hall/CRC} 
}

@article{montufar2015when,
    author = {Mont\'{u}far, Guido and Morton, Jason},
    title = {When {D}oes a {M}ixture of {P}roducts {C}ontain a {P}roduct of {M}ixtures?},
    journal = {SIAM Journal on Discrete Mathematics},
    year = {2015},
}

@article{JMLR:v9:krause08a,
    author  = {Andreas Krause and Ajit Singh and Carlos Guestrin},
    title   = {{Near-optimal sensor placements in Gaussian processes: Theory, efficient algorithms and empirical studies}},
    journal = {Journal of Machine Learning Research},
    year    = {2008},
}

@article{sviridenko2017optimal,
    title={Optimal {a}pproximation for {s}ubmodular and {s}upermodular {o}ptimization with {b}ounded {c}urvature},
    author={Sviridenko, Maxim and Vondr{\'a}k, Jan and Ward, Justin},
    journal={Mathematics of Operations Research},
    year={2017},
}

@inproceedings{chen2018weakly,
    title={Weakly {s}ubmodular {m}aximization {b}eyond {c}ardinality {c}onstraints: {d}oes {r}andomization {h}elp {g}reedy?},
    author={Chen, Lin and Feldman, Moran and Karbasi, Amin},
    booktitle={International Conference on Machine Learning},
    pages={804--813},
    year={2018},
}

@inproceedings{10.5555/2634074.2634180,
    author = {Buchbinder, Niv and Feldman, Moran and Naor, Joseph (Seffi) and Schwartz, Roy},
    title = {Submodular {M}aximization with {C}ardinality {C}onstraints},
    year = {2014},
    booktitle = {Proceedings of the Twenty-Fifth Annual ACM-SIAM Symposium on Discrete Algorithms}
}

@inproceedings{bian2017guarantees,
    title={Guarantees for {G}reedy {M}aximization of {N}on-{S}ubmodular {F}unctions with {A}pplications},
    author={Bian, Andrew An and Buhmann, Joachim M and Krause, Andreas and Tschiatschek, Sebastian},
    booktitle={International Conference on Machine Learning},
    year={2017},
}

@article{gatmiry2018non,
    title={Non-{S}ubmodular {F}unction {M}aximization {S}ubject to a {M}atroid {C}onstraint, with {A}pplications},
    author={Gatmiry, Khashayar and Gomez-Rodriguez, Manuel},
    journal={arXiv preprint arXiv:1811.07863},
    year={2018}, 
    note={See also ``The Network Visibility Problem'', ACM Transactions on Information Systems, Volume 40, Issue 2, Article No.: 22, pp 1–42, 2021} 
}

@article{CONFORTI1984251,
    title = {Submodular {S}et {F}unctions, {M}atroids and the {G}reedy {A}lgorithm: {T}ight {W}orst-{C}ase {B}ounds and {S}ome {G}eneralizations of the {R}ado-{E}dmonds {T}heorem},
    journal = {Discrete Applied Mathematics},
    year = {1984},
    author = {Michele Conforti and Gérard Cornuéjols},
}

@inproceedings{10.5555/1347082.1347101,
    author = {Du, Ding-Zhu and Graham, Ronald L. and Pardalos, Panos M. and Wan, Peng-Jun and Wu, Weili and Zhao, Wenbo},
    title = {{Analysis of Greedy Approximations with Non-Submodular Potential Functions}},
    year = {2008},
    publisher = {Society for Industrial and Applied Mathematics},
    address = {USA},
    booktitle = {Proceedings of the 19th Annual ACM-SIAM Symposium on Discrete Algorithms},
    pages = {167–175},
    numpages = {9},
    location = {San Francisco, California},
    series = {SODA '08}
}

@inproceedings{filmus2012tight,
  title={A tight combinatorial algorithm for submodular maximization subject to a matroid constraint},
  author={Filmus, Yuval and Ward, Justin},
  booktitle={2012 IEEE 53rd Annual Symposium on Foundations of Computer Science},
  pages={659--668},
  year={2012},
  organization={IEEE}
}

@inproceedings{NIPS2016_81c8727c,
 author = {Horel, Thibaut and Singer, Yaron},
 booktitle = {Advances in Neural Information Processing Systems},
 noeditor = {D. Lee and M. Sugiyama and U. Luxburg and I. Guyon and R. Garnett},
 pages = {},
 publisher = {Curran Associates, Inc.},
 title = {Maximization of {A}pproximately {S}ubmodular {F}unctions},
 nourl = {https://proceedings.neurips.cc/paper/2016/file/81c8727c62e800be708dbf37c4695dff-Paper.pdf},
 volume = {29},
 year = {2016}
}

@article{Nemhauser1978AnAO,
  title={An analysis of approximations for maximizing submodular set functions--{I}},
  author={George L. Nemhauser and Laurence A. Wolsey and Marshall L. Fisher},
  journal={Mathematical Programming},
  year={1978},
  volume={14},
  pages={265-294}
}

@inproceedings{das2011submodular, 
author = {Das, Abhimanyu and Kempe, David}, 
title = {Submodular meets spectral: Greedy algorithms for subset selection, sparse approximation and dictionary Selection}, 
year = {2011}, 
noisbn = {9781450306195}, 
nopublisher = {Omnipress}, 
noaddress = {Madison, WI, USA}, 
booktitle = {Proceedings of the 28th International Conference on International Conference on Machine Learning}, 
pages = {1057–1064}, 
nonumpages = {8}, 
nolocation = {Bellevue, Washington, USA}, 
noseries = {ICML'11} 
}

@article{allman2015tensors,
title = {Tensors of {N}on-{N}egative {R}ank {T}wo},
journal = {Linear Algebra and its Applications},
volume = {473},
pages = {37-53},
year = {2015},
note = {Special issue on Statistics},
noissn = {0024-3795},
nodoi = {https://doi.org/10.1016/j.laa.2013.10.046},
url = {https://www.sciencedirect.com/science/article/pii/S0024379513006812},
author = {Elizabeth S. Allman and John A. Rhodes and Bernd Sturmfels and Piotr Zwiernik}
}

@article{Kuipers2010AGO,
  title={A {G}eneralization of the {S}hapley–{I}chiishi {R}esult},
  author={Jeroen Kuipers and Dries Vermeulen and Mark Voorneveld},
  journal={International Journal of Game Theory},
  year={2010},
}

@article{Clinescu2011MaximizingAM,
  title={Maximizing a monotone submodular function subject to a matroid constraint},
  author={Gruia Călinescu and Chandra Chekuri and Martin P{\'a}l and Jan Vondr{\'a}k},
  journal={SIAM J. Comput.},
  year={2011},
  volume={40},
  pages={1740-1766}
}

@article{JMLR:v23:20-1424,
  author  = {Rishi Sonthalia and Anna C. Gilbert},
  title   = {Project and {F}orget: {S}olving {l}arge-{s}cale {m}etric {c}onstrained {p}roblems},
  journal = {Journal of Machine Learning Research},
  year    = {2022},
  volume  = {23},
  number  = {326},
  pages   = {1--54},
  url     = {http://jmlr.org/papers/v23/20-1424.html}
}

@inproceedings{10.5555/3172077.3172251,
author = {Qian, Chao and Shi, Jing-Cheng and Yu, Yang and Tang, Ke and Zhou, Zhi-Hua},
title = {Optimizing ratio of monotone set functions},
year = {2017},
noisbn = {9780999241103},
publisher = {AAAI Press},
booktitle = {Proceedings of the 26th International Joint Conference on Artificial Intelligence},
pages = {2606–2612},
numpages = {7},
location = {Melbourne, Australia},
series = {IJCAI'17}
}

@InProceedings{pmlr-v48-baib16,
  title = 	 {Algorithms for {O}ptimizing the {R}atio of {S}ubmodular {F}unctions},
  author = 	 {Bai, Wenruo and Iyer, Rishabh and Wei, Kai and Bilmes, Jeff},
  booktitle = 	 {Proceedings of The 33rd International Conference on Machine Learning},
  pages = 	 {2751--2759},
  year = 	 {2016},
  editor = 	 {Balcan, Maria Florina and Weinberger, Kilian Q.},
  volume = 	 {48},
  series = 	 {Proceedings of Machine Learning Research},
  address = 	 {New York, New York, USA},
  month = 	 {20--22 Jun},
  publisher =    {PMLR},
  url = 	 {https://proceedings.mlr.press/v48/baib16.html}
}

@article{Wang2019MinimizingRO,
  title={Minimizing ratio of monotone non-submodular functions},
  author={Yijing Wang and Dachuan Xu and Yanjun Jiang and Dongmei Zhang},
  journal={Journal of the Operations Research Society of China},
  year={2019},
  pages={1-11}
}

@InProceedings{pmlr-v84-bogunovic18a,
  title = 	 {Robust {M}aximization of {N}on-{S}ubmodular {O}bjectives},
  author = 	 {Bogunovic, Ilija and Zhao, Junyao and Cevher, Volkan},
  booktitle = 	 {Proceedings of the Twenty-First International Conference on Artificial Intelligence and Statistics},
  pages = 	 {890--899},
  year = 	 {2018},
  editor = 	 {Storkey, Amos and Perez-Cruz, Fernando},
  volume = 	 {84},
  series = 	 {Proceedings of Machine Learning Research},
  month = 	 {09--11 Apr},
  publisher =    {PMLR},
  url = 	 {https://proceedings.mlr.press/v84/bogunovic18a.html}
}

@article{Seigal2018MixturesAP,
  title={Mixtures and products in two graphical models},
  author={Anna Seigal and Guido Mont{\'u}far},
  journal={Journal of Algebraic Statistics},
  year={2018}
}

@incollection{cueto2010geometry,
author = {Maria Angélica Cueto and Jason Morton and Bernd Sturmfels},
title = {{Geometry of the restricted Boltzmann machine}},
series = {Contemporary mathematics},
volume = {516},
nodoi = {10.1090/conm/516/10172},
booktitle = {Algebraic methods in statistics and probability\,:\,volume 2},
editor = {Marlos A. G. Viana and Henry P. Wynn},
publisher = {American Mathematical Society},
address = {Providence, R.I.},
year = {2010},
pages = {135--153},
noisbn={978-0-8218-4891-3},
}

@article{assarf2014computing,
	Author = {Assarf, Benjamin and Gawrilow, Ewgenij and Herr, Katrin and Joswig, Michael and Lorenz, Benjamin and Paffenholz, Andreas and Rehn, Thomas},
	Da = {2017/03/01},
	noDoi = {10.1007/s12532-016-0104-z},
	Id = {Assarf2017},
	noIsbn = {1867-2957},
	Journal = {Mathematical Programming Computation},
	Number = {1},
	Pages = {1--38},
	Title = {Computing {C}onvex {H}ulls and {C}ounting {I}nteger {P}oints with {P}olymake},
	Ty = {JOUR},
	Url = {https://doi.org/10.1007/s12532-016-0104-z},
	Volume = {9},
	Year = {2017}
 }

@article{Gawrilow2000,
    author="Gawrilow, Ewgenij and Joswig, Michael",
    editor="Kalai, Gil and Ziegler, G{\"u}nter M.",
    title="Polymake: A {F}ramework for {A}nalyzing {C}onvex {P}olytopes",
    journal="Polytopes --- Combinatorics and Computation",
    year="2000",
    publisher="Birkh{\"a}user Basel",
    address="Basel",
    pages="43--73",
}

@manual{sage,
  Key          = {Sage},
  Author       = {William A. Stein and others},
  Organization = {The Sage Development Team},
  Title        = {{S}age {M}athematics {S}oftware ({V}ersion 9.4)},
  note         = {{\tt http://www.sagemath.org}},
  Year         = {2023},
}

\newpage 

\appendix 

\section{Notation}  

We summarize our notation in the following table. 

\begin{table}[H]
\centering
\begin{tabular}{p{.6in}|p{4.5in}}
    \xrowht[()]{10pt}
   $[n]$            &  $\{1,\ldots, n\}$. 
   \\  \hline\xrowht[()]{10pt}    
   $(X, \pi)$     & A poset $X$ with partial order defined by a tuple $\pi$. 
   \\ \hline \xrowht[()]{10pt}
   $f$ & A function from a poset $(X, \pi)$ to $\mathbb{R}$. \\ \hline
   \xrowht[()]{10pt}
   $\mathcal{L_{\pi}}$ & The cone of $\pi$-supermodular functions. 
   \\
   \hline\xrowht[()]{10pt} 
   $\pi$              & Tuple of $n$ linear orders $\pi = (\pi_1, \ldots, \pi_n)$, see Definition~\ref{def:pi_supermodularity}.  \\ \hline
   \xrowht[()]{10pt}
   $A^{(ij)}$ & The $(ij)$ elementary imset matrix, a $2^{n-2} \times 2^n$ matrix that collects the imset inequalities for each $x\in\{0,1\}^{n-2}$, see Definition~\ref{def:facets}. 
   \\ \hline\xrowht[()]{10pt}
   $\tau$ & Sign vector $\tau=(\tau_1,\ldots, \tau_n)$ of a tuple of linear orders $\pi$, see Definition~\ref{def:the_sign_vector}.
   \\ \hline\xrowht[()]{10pt} 
   $\xi$ & Vector in $\{-1,0,1\}^{\binom{n}{2}}$, described in Definition~\ref{def:xi-cone}. 
   \end{tabular}

\bigskip 

   Notions of curvature and submodularity ratio. 
   
   \begin{tabular}{p{.6in}|p{4.5in}}
   \xrowht[()]{10pt} 
   $\hat{\alpha}$ & The \emph{total curvature} of a normalized, monotone submodular function $f$ is $\hat{\alpha} := \max_{e \in \Omega}\frac{\Delta(e|\emptyset) - \Delta(e|V\setminus \{e\})}{\Delta(e|\emptyset)}$, where $\Omega = \{ e \in V \colon \Delta(e|\emptyset) > 0\}$ \\ \hline \xrowht[()]{10pt} 
   $\gamma_{X,m}$ & The \emph{submodularity ratio} of a non-negative monotone function with respect to a set $X$ and integer $m$ is $\gamma_{X,m} := \min_{T \subset X, S \subset V, |S| \le m, S \cap T = \emptyset} \frac{\sum_{e \in S}\Delta(e|T)}{\Delta(S|T)}$. Subscripts dropped for $X = V$, $m = k$ \\ \hline \xrowht[()]{10pt} 
   $\alpha$ &  The \emph{generalized curvature} of a non-negative monotone set function is the smallest $\alpha$ s.t.\ for all $T, S \in 2^V$ and $e \in S \setminus T$,
   $\displaystyle \Delta(e | (S \setminus \{e\}) \cup T) \ge (1-\alpha)\Delta(e | S \setminus \{e\})$ \\ \hline \xrowht[()]{10pt} 
   $\tilde{\alpha}$ & The \emph{generalized inverse curvature} of a non-negative set function $f$ is the smallest $\tilde{\alpha}^f$ s.t. for all $T, S \in 2^V$ and $e \in S \setminus T$, $\displaystyle \Delta(e | S \setminus \{e\}) \ge (1-\tilde{\alpha}^f) \Delta(e | (S \setminus \{e\}) \cup T).$ \\ \hline \xrowht[()]{10pt} 
   $\hat{c}$ & The \emph{curvature} of $f$ with respect to $X$ is $\hat{c}^f(X) := 1 - \frac{\sum_{e \in X} (f(X) - f(X\setminus \{e\})}{\sum_{e\in X} f(\{e\})}$  \\
\end{tabular}
\caption{Notation used in the paper}
\end{table}

\section{Relations of Supermodular Rank and Probability Models}
\label{app:relations-probability-models}
Probabilistic graphical models are defined by imposing conditional independence relations between the variables that are encoded by a graph \cite{lauritzen1996graphical}. Supermodularity (in)equalities arise naturally in such models. 
Our investigations in Section~\ref{sec:minkowski-sums-of-supermodular-cones} are motivated in part and addresses questions that have been raised in the context of graphical models in \cite{Seigal2018MixturesAP}.
We explain this connection. 

\paragraph{Conditional independence and modularity}  
Consider finite-valued random variables, denoted $X_i$, which take values denoted by lower case letters $x_i$. 
Two random variables $X_1$ and $X_2$ are conditionally independent given a third variable $X_3$ if, for any fixed value $x_3$ that occurs with positive probability, the matrix of conditional joint probabilities $p(x_1,x_2|x_3) = p(x_1,x_2,x_3)/p(x_3)$ 
factorizes as a product of two vectors of conditional marginal probabilities. 
This means that the matrix of conditional joint probabilities has rank one. The vanishing of the $2 \times 2$ minors is the requirement that any submatrix obtained by looking at two rows and two columns has determinant zero, $p(x_1,x_2|x_3) p(x'_1,x'_2|x_3) - p(x_1,x'_2|x_3) p(x'_1,x_2|x_3) = 0$, for any two rows $x_1,x'_1$ and any two columns $x_2,x'_2$. 
The rank-one condition is rewritten in terms of log probabilities as 
$\log p(x_1,x_2,x_3)+\log p(x'_1,x'_2,x_3) = \log p(x_1,x'_2,x_3)+\log p(x'_1,x_2,x_3)$,
for all $x_1,x_1'$ in the range of values of $X_1$, all $x_2,x_2'$ in the range of values of $X_2$, and all $x_3$ in the range of values of $X_3$. 
Thus, with an appropriate partial order on the sample space, a conditional independence statement corresponds to modularity equations for log probabilities. 

\paragraph{Latent variables and supermodularity} 

If some of the random variables are \emph{hidden} (or \emph{latent}), characterizing the visible marginals in terms of (in)equality relations between visible margins becomes a challenging problem \citep[see, e.g.,][]{garcia2005algebraic,
allman2015tensors, 
Zwiernik2015SemialgebraicSA, 
montufar2015when, 
qi2016semialgebraic, 
evans2018margins, 
Seigal2018MixturesAP}. 
In several known cases, such descriptions involve conditional independence inequalities that correspond to submodularity or supermodularity inequalities for log probabilities. 
For instance, \cite{allman2015tensors} studied $n$ discrete visible variables that are conditionally independent given a binary hidden variable. This is known as the $2$-mixture of a $n$-variable independence model, and denoted by $\mathcal{M}_{n,2}$. The article \cite{allman2015tensors} shows that
the visible marginals of $\mathcal{M}_{n,2}$ are characterized by the vanishing of certain $3\times 3$ minors (these are equalities) 
and conditional independence inequality relations that impose that the log probabilities are $\pi$-supermodular for some partial order $\pi$. 
Thus, the set of log probabilities defined by inequalities of $\mathcal{M}_{n,2}$, discarding the equalities, gives a union of $\pi$-supermodular cones. 

\paragraph{Restricted Boltzmann machines and sums of supermodular cones} 
A prominent graphical model is the restricted Boltzmann machine (RBM) \citep{Smolensky1986,Hinton2002}. 
The model $\RBM_{n,\ell}$ has $n$ visible and $\ell$ hidden variables, and defines the probability distributions of the $n$ visible variables as Hadamard (entrywise) products of any $\ell$ probability distributions belonging to $\mathcal{M}_{n,2}$. 
Characterizing the visible marginals represented by this model has been a topic of interest, see for instance \citep{6796877,NIPS2011_8e98d81f,NIPS2013_7bb06076}. In particular, \cite{cueto2010geometry,doi:10.1137/16M1077489} studied the dimension of this model, and \cite{montufar2015when,MONTUFAR2017531} investigated certain inequalities satisfied by the visible marginals. 

Proposition~\ref{prop:supermodular-rank-submodular} implies that the model $\RBM_{n,\ell}$ is not a universal approximator whenever $\ell < \lceil \log_2 n\rceil+1$. 
This does not give new non-trivial bounds for the minimum size of a universal approximator for $n\geq 4$, since the number of parameters of the model, $(n+1)(\ell+1)-1 $, is smaller than the dimension of the space, $2^n - 1$. 
However, the result shows that any probability distributions on $\{0,1\}^n$ whose logarithm is in the interior of the submodular cone requires at least $\ell = \lceil \log_2(n)\rceil$ to lie in $\RBM_{n,\ell}$. 
This complements previous results based on polyhedral sets called mode poset probability polytopes \citep{montufar2015when,modeposet2015}.

\end{document}